\documentclass{amsart}

\usepackage[utf8]{inputenc} 
\usepackage[T1]{fontenc} 

\usepackage{amsmath, amsthm, amssymb, mathrsfs,amsrefs}
\usepackage{enumerate}
\usepackage{tikz-cd}
\usetikzlibrary{arrows}

\usetikzlibrary{matrix,arrows,decorations.pathmorphing}

 \let\temp\phi
\let\phi\varphi
\let\varphi\temp

\newtheorem{innercustomtheorem}{Theorem}
\newenvironment{customtheorem}[1]
  {\renewcommand\theinnercustomtheorem{#1}\innercustomtheorem}
  {\endinnercustomtheorem}

\newcommand{\acts}{\curvearrowright}
\DeclareMathOperator{\Aut}{Aut}
\newcommand{\cl}[1]{\overline{#1}}

\newcommand{\cW}{\mathcal{W}}

\newcommand{\C}{\mathbb{C}}

\newcommand{\T}{\mathbb{T}}

\newcommand{\Z}{\mathbb{Z}}

\newcommand{\ran}{\text{ran }}

\DeclareMathOperator{\Span}{span}

\newcommand{\calA}{\mathcal{A}}
\newcommand{\calB}{\mathcal{B}}
\newcommand{\calC}{\mathcal{C}}
\newcommand{\calD}{\mathcal{D}}

\newcommand{\calF}{\mathcal{F}}

\newcommand{\calH}{\mathcal{H}}
\newcommand{\calI}{\mathcal{I}}
\newcommand{\calJ}{\mathcal{J}}
\newcommand{\calK}{\mathcal{K}}

\newcommand{\calM}{\mathcal{M}}

\newcommand{\calO}{\mathcal{O}}

\newcommand{\calS}{\mathcal{S}}
\newcommand{\calT}{\mathcal{T}}
\newcommand{\calU}{\mathcal{U}}

\newcommand{\calW}{\mathcal{W}}

\newcommand{\Angle}[1]{\left\langle #1 \right\rangle}

\DeclareMathOperator{\id}{id}

\newcommand{\bC}{\mathbb{C}}
\newcommand{\bofh}{\mathcal{B}(\mathcal{H})}

\theoremstyle{plain}
\newtheorem{lemma}{Lemma}[section]
\newtheorem{theorem}[lemma]{Theorem}

\newtheorem{proposition}[lemma]{Proposition}
\newtheorem{corollary}[lemma]{Corollary}

\theoremstyle{definition}

\newtheorem{example}[lemma]{Example}
\newtheorem{remark}[lemma]{Remark}
\newtheorem{definition}[lemma]{Definition}

\newtheorem{problem}[lemma]{Problem}

\newcommand{\frakg}{\mathfrak{g}}
\newcommand{\frakh}{\mathfrak{h}}
\newcommand{\frakz}{\mathfrak{z}}
\newcommand{\env}{\text{env}}
\newcommand{\ess}{\text{ess}}
\newcommand{\stmin}{\text{C*-min}}

\title[Crossed products of operator systems]{Crossed products of operator systems}
\author{Samuel J. Harris}
\address{University of Waterloo \\
Department of Pure Mathematics \\
Waterloo, Ontario \\
Canada  N2L 3G1}
\email{sj2harri@uwaterloo.ca}
\author{Se-Jin Kim}
\address{University of Waterloo \\
Department of Pure Mathematics \\
Waterloo, Ontario \\
Canada  N2L 3G1}
\email{s362kim@uwaterloo.ca}

\begin{document}
\begin{abstract}
	In this paper we introduce the crossed product construction for a discrete group action on an operator system. In analogy to the work of E. Katsoulis and C. Ramsey, we describe three canonical crossed products arising from such a dynamical system. We describe how these crossed product constructions behave under $G$-equivariant maps, tensor products, and the canonical $C^*$-covers. We show that hyperrigidity is preserved under two of the three crossed products. Finally, using A. Kavruk's notion of an operator system that detects $C^*$-nuclearity, we give a negative answer to a question on operator algebra crossed products posed by Katsoulis and Ramsey.
	
\end{abstract}

\maketitle

\tableofcontents

\section{Introduction}

Crossed products for $C^{\ast}$-algebras provide a recipe for constructing new $C^{\ast}$-algebras from dynamical information. This construction has had many fruitful applications in the theory of operator algebras as well as dynamical systems. The work of Kalantar and Kennedy \cite{kalantarkennedy} as well as the work of Kawabe \cite{kawabe} suggest that it is possible to study the dynamics of a dynamical system $(X,G)$ or a group $G$ by looking at the $G$-injective envelope of the associated $C^{\ast}$-algebra $I_G(C(X))$ introduced by Hamana \cite{hamanadynsys}. However, the $C^{\ast}$-algebra $I_G(C(X))$ is first and foremost constructed as a $G$-operator system--an operator system with an associated $G$-action given by unital complete order isomorphisms. This approach suggests that many of the results of $C^{\ast}$-dynamics may be phrased more naturally in the category of operator systems. One of the goals of this paper is to take a first step in this direction by introducing the notion of an operator system crossed product.

	In \cite{katsoulisramsey}, E. Katsoulis and C. Ramsey introduce a variant on crossed products of $C^{\ast}$-algebras for norm closed subalgebras of $B(\calH)$ which are not necessarily closed under involution. Such algebras are called \emph{non-selfadjoint operator algebras}. Crossed products of a dynamical system consisting of a (non-selfadjoint) approximately unital operator algebra $\calA$ along with a continuous $G$-action $\alpha$ acting by unital completely isometric isomorphisms are constructed by taking the norm closure of $C_c(G,\calA)$ within the crossed products of  various $G$-equivariant $C^{\ast}$-covers of $\calA$. Their work is partially motivated by the Hao-Ng isomorphism problem.
  \begin{problem}
  \emph{(Hao-Ng Isomorphism Problem)}
    Let $(X,\calC)$ be a non-degenerate $C^{\ast}$-correpondence, and let $\alpha: G \to \Aut(X,\calC)$ be the generalized gauge action of a locally compact group $G$. Do we have the isomorphism
    \begin{align*}
      \calO_X \rtimes_{\alpha} G \simeq \calO_{X \rtimes_\alpha G}\;?
    \end{align*}
  \end{problem}
  See \cite{BKQR} and \cite{KQR} for more information on this problem. Hao and Ng show that this isomorphism holds when $G$ is an amenable group \cite[Theorem 2.10]{hao-ng}. One can also ask whether the Hao-Ng isomorphism problem has a positive answer when using reduced crossed products instead of full crossed products. Katsoulis applies the crossed product construction of non-selfadjoint algebras to give an affirmative answer to the reduced version of the Hao-Ng isomorphism problem in the case of discrete groups \cite{katsoulis17}. One way to ensure that the Hao-Ng isomorphism problem has a positive answer for full crossed products is by ensuring that there is only one possibility for the full crossed product of an (approximately unital) operator algebra by a locally compact group.
  \begin{definition}[(Katsoulis-Ramsey, \cite{katsoulisramsey})]
    If $(\calA,G,\alpha)$ is an operator algebraic dynamical system, then define $\calA \rtimes_{C^*_{\env}(\calA),\alpha} G$ and $\calA \rtimes_{C^*_u(\calA),\alpha} G$ to be the norm closure of $C_c(G,\calA)$ in $C^*_{\env}(\calA) \rtimes G$ and $C^*_u(\calA) \rtimes G$, respectively.
  \end{definition}
  The monograph of Katsoulis and Ramsey ends with a list of open problems regarding operator algebra crossed products. Problem 1 asks whether the identity $C_{\env}^*(\calA \rtimes_{C_u^*(\calA),\alpha} G)=C_{\env}^*(\calA) \rtimes_{\alpha} G$ holds for all operator algebra dynamical systems.  If the identity above was true for all operator algebraic dynamical systems, then the Hao-Ng isomorphism theorem would have a positive answer for all $C^*$-correspondences and all locally compact groups \cite[Chapter 7]{katsoulisramsey}. In a similar direction, Problem 2 asks whether the images of $C_c(G,\calA)$ in $C^*_{\env}(\calA) \rtimes G$ and $C^*_u(\calA) \rtimes G$ are canonically completely isometrically isomorphic. This would imply that there is only one possible choice for a full crossed product for an operator algebra dynamical system.

In this paper, we construct several operator system crossed products for discrete group actions. These crossed products are generalizations of the operator algebraic crossed product introduced by Katsoulis and Ramsey in the case of discrete group actions (see Remarks \ref{remark: connection to operator algebra, reduced} and \ref{remark: connection to operator algebra, full}). The work of Farenick et. al \cite{DiscGroup} on operator systems associated to discrete groups suggests that, instead of considering operator systems of the form $C_c(G,\calS)$ inside a crossed product $C^{\ast}$-algebra, one can instead consider a generating set $\frakg$ (usually taken to be finite) of $G$ which contains the identity $e$ and is inverse-closed. In this way, one can generate a crossed product operator system $\calS \rtimes \frakg = C_c(\frakg, \calS)$. This approach is attractive for two reasons: firstly, in the case of finite dimensional operator systems, the dual of this matrix ordered space admits an operator system structure. Moreover, it is known that classification of finite dimensional operator systems has lower complexity than the classification of separable operator systems \cite{ACKKLS}. Since we only need to deal with the order structure on our vector space, and since our group is always assumed to be discrete, many of the results of Katsoulis and Ramsey for discrete group actions on unital operator algebras have a generalization to the operator system case with simpler proofs.

There is an important reason for our restriction to operator system dynamical systems where the group is discrete. If the group is not discrete, then the resulting crossed product may fail to have an order unit, and hence it may fail to be an operator system. This follows since the full group $C^*$-algebra of a non-discrete, locally compact group $G$ may not have a unit. One solution to this problem is to approach operator system dynamical systems involving general locally compact groups by allowing for a generalization of operator systems where no order unit is present. The authors plan to address this problem in future work.

We also apply the theory of operator system nuclearity detectors to operator algebra crossed products. A nuclearity detector $\calS$ is an operator system with the property that, for any unital $C^{\ast}$-algebra $\calC$, the identity $\calS \otimes_{\min} \calC = \calS \otimes_{\max} \calC$ holds if and only if $\calC$ is nuclear \cite{kavruk15}. From a nuclearity detector $\calS$, we construct the non-selfadjoint operator algebra $\calU(\calS)$ introduced by D. Blecher and C. Le Merdy \cite{blecherlemerdy}. Assuming that $\calS$ is a nuclearity detector, $\calU(\calS)$ has the property that, if $\calC$ is any (not necessarily unital) $C^{\ast}$-algebra and $\calU(\calS) \otimes_{\min} \calC = \calU(\calS) \otimes_{\max} \calC$ then $\calC$ must be nuclear. The smallest known nuclearity detector is the four-dimensional operator system $\calW_{3,2} \subseteq \bigoplus_{k=1}^3 M_2(\C)$, which will be considered in Sections \S4, \S5 and \S6. Using this nuclearity detector, we have the following theorem:
  \begin{customtheorem}{\ref{KRexamplelocallycompactgroup}}
  \label{Theorem: intro counterexample problem 2 KR}
    Suppose that $G$ is a locally compact group such that $C^*_\lambda(G)$ admits a tracial state. Let $\calA := \calU(\cW_{3,2})$ be the operator subalgebra of $M_2\left(\bigoplus_{k=1}^3 M_2(\C)\right)$ endowed with the trivial $G$-action $\id: G \acts \calA$. The following are equivalent:
    \begin{enumerate}
      \item $\calA \rtimes_{C^*_{env}(\calA),\id} G = \calA \rtimes_{C^*_u(\calA),\id} G$.
      \item The group $G$ is amenable.
    \end{enumerate}
  \end{customtheorem}
  Theorem \ref{Theorem: intro counterexample problem 2 KR} provides a counterexample to Problem 2 in \cite{katsoulisramsey} for a large class of locally compact groups. The operator algebra $\calA=\calU(\calW_{3,2})$ is surprisingly tame. Because $\calA$ is constructed from Kavruk's nuclearity detector, which is four-dimensional, the operator algebra we obtain satisfies $\dim(\calA) = 5$. Moreover, $\calA$ is hyperrigid in its $C^{\ast}$-envelope (see Theorem \ref{Theorem: hyperrigidity of counterexample}).

  The counterexample of Theorem \ref{Theorem: intro counterexample problem 2 KR} allows us to give a counterexample to Problem 1 of \cite{katsoulisramsey}.

  \begin{customtheorem}{\ref{KRproblem1}}
  \label{Theorem: intro counterexample problem 1 KR}
  Let $\calA=\calU(\calW_{3,2})$. If $G$ is a discrete group, then the following are equivalent:
  \begin{enumerate}
  \item
  We have the identity $C_{\env}^*(\calA) \rtimes_{\id} G=C_{\env}^*(\calA \rtimes_{\id} G)$.
  \item
  The group $G$ is amenable.
  \end{enumerate}
\end{customtheorem}

We note that the operator algebra $\calA$ does not provide a counterexample to the Hao-Ng isomorphism problem for full crossed products. Indeed, since our work was first posted, Katsoulis and Ramsey have posted a new paper \cite{katsoulisramsey18}, addressing the relation between Problems 1 and 2 of their monograph \cite{katsoulisramsey} and the Hao-Ng isomorphism problem. Since $\calA$ is not the tensor algebra of a $C^*$-correspondence, it cannot provide a counterexample to the isomorphism problem. On the other hand, for hyperrigid $C^*$-correspondences, Katsoulis and Ramsey show \cite[Theorem 4.9]{katsoulisramsey18} that Problems 1 and 2 are equivalent for the tensor algebra of the $C^*$-correspondence. Moreover, in this setting, these problems are equivalent to the Hao-Ng isomorphism problem. We refer the reader to \cite{katsoulisramsey18} for more information.

The paper is organized as follows. In Section \S2 we recall some basic facts about $C^*$-envelopes, operator algebra crossed products, and operator system theory that we will use throughout the paper. In Section \S3 we develop the theory of reduced crossed products for operator systems; we also show that the reduced crossed product is independent of the choice of $C^*$-cover chosen for the operator system.  Section \S4 concerns the more subtle theory of full crossed products of operator systems. We focus on two natural choices for $C^*$-covers for the full crossed product: the $C^*$-envelope of $\calS$ (resulting in the full enveloping crossed product) and the univeral $C^*$-algebra of $\calS$ (resulting in the full crossed product). We show that the reduced and full enveloping crossed products agree when the group $G$ is amenable. We turn to nuclearity detectors in Section \S5, allowing us to give counterexamples to the first two problems of the monograph of Katsoulis and Ramsey \cite{katsoulisramsey}. Finally, we show in Section \S6 that the operator algebra $\calU(\calW_{3,2})$ is hyperrigid in its $C^*$-envelope, which shows that the identity $C_{\env}^*(\calA \rtimes_{\alpha} G)=C_{\env}^*(\calA) \rtimes_{\alpha} G$ may fail even if $\calA$ is hyperrigid in its $C^*$-envelope.

\section{Preliminaries}

\subsection{$C^{\ast}$-envelopes and maximal dilations}

We assume that the reader is familiar with the basic definitions and theory of abstract operator systems.  See \cite[Chapter 13]{paulsenbook} for a thorough introduction to this topic.

Let $\calS$ be an operator system. A $C^{\ast}$\textit{-cover} for $\calS$ is a pair $(\calC, \rho)$, where $\calC$ is a $C^{\ast}$-algebra and $\rho: \calS \hookrightarrow \calC$ is a unital complete order isomorphism such that $C^*(\rho(\calS))=\calC$. If $(\calC,\rho)$ is a $C^*$-cover for $\calS$ and $\calI$ is an ideal in $\calC$, we say that $\calI$ is a \textit{boundary ideal} if the restriction of the canonical quotient map $q:\calC \to \calC/\calI$ to $\calS$ is a complete order embedding \cite{arveson69}. The \textit{Shilov ideal} $\calJ_{\calS}$ corresponding to the $C^*$-cover $(\calC,\rho)$ is the maximal boundary ideal; that is, whenever $\calI$ is a boundary ideal for $\calS$ in $(\calC,\rho)$, we have $\calI \subseteq \calJ_{\calS}$. M. Hamana showed that the Shilov ideal for $\calS$ in $(\calC,\rho)$ always exists \cite{hamanaopsys}. Moreover, in \cite{hamanaopsys}, it is shown that every operator system $\calS$ admits a unique $C^{\ast}$-cover $(C^*_{\env}(\calS),\iota)$, called the $C^{\ast}$\textit{-envelope} of $\calS$, satisfying the following universal property: whenever $(\calC,\rho)$ is a $C^{\ast}$-cover of $\calS$, there is a unique $*$-epimorphism $\pi: \calC \to C^*_{\env}(\calS)$ for which the diagram
\begin{center}
  \begin{tikzcd}
    \calC \arrow{rr}{\pi} & & C^*_{\env}(\calS) \\
    & \calS \arrow{ul}{\rho} \arrow{ur}{\iota} &
  \end{tikzcd}
\end{center}
commutes. In this setting, the Shilov ideal $\calJ_{\calS}$ is given precisely by the kernel of the map $\pi$ \cite{hamanaopsys}. In other words, $\calC/\calJ_{\calS} \simeq C_{env}^*(\calS)$. The proof in \cite{hamanaopsys} that a $C^{\ast}$-envelope always exists uses the injective envelope of an operator system. Although this construction is useful in many respects, the construction of the $C^{\ast}$-envelope on which we wish to concentrate in this paper is the construction given by maximal dilations. Given a unital, completely positive (ucp) map $\phi: \calS \to B(\calH)$, we say that a representation $\rho: \calS \to B(\calK)$ is a \textit{dilation} of $\phi$ if there is an isometry $V: \calH \hookrightarrow \calK$ for which $V\phi(x) = \rho(x)V$ for all $x \in \calS$. We will always assume without loss of generality that $\calH \subset \calK$. In this way, we may always set $\calK = \calH \oplus \calH^\perp$ and represent $\rho(x)$ as the block $2 \times 2$-matrix
\begin{align*}
  \rho(x) = \left[\begin{array}{cc} \phi(x) & a_x \\ c_x & b_x \end{array} \right]
\end{align*}
for some $a_x \in B(\calH^\perp,\calH)$, $b_x \in B(\calH^\perp)$ and $c_x \in B(\calH,\calH^{\perp})$. Note that $c_x=a_{x^*}$. The compression to the $(2,2)$-corner of $\rho(x)$ is also a ucp map $\phi': \calS \to B(\calH^\perp)$. We say that the dilation $\rho$ of $\phi$ is \textit{trivial} if $\rho = \phi \oplus \phi'$; that is, $a_x=0$ for all $x \in \calS$. A ucp map $\phi: \calS \to B(\calH)$ is \textit{maximal} if the only dilations of $\phi$ are trivial dilations. For an operator system $\calS$ and a ucp map $\phi: \calS \to B(\calH)$, we say that $\phi$ has the \textit{unique extension property} if there is a unique ucp extension of $\phi$ to $C_{\env}^*(\calS)$ and the unique extension is a $*$-homomorphism. Dritschel and McCollough (see \cite[Theorem 2.5]{arveson08} and \cite{dritschelmccullough}) show that a ucp map $\phi: \calS \to B(\calH)$ is maximal if and only if $\phi$ satisfies the unique extension property. In an unpublished work of Arveson, it is shown that every representation of an operator system has a maximal dilation \cite[Theorem 1.3]{arveson03}, and that if $\phi: \calS \to B(\calH)$ is maximal, then $\phi(\calS)$ necessarily generates the $C^{\ast}$-envelope of $\calS$ \cite[Corollary 3.3]{arveson03}. We will use all of these facts freely throughout this paper.

\subsection{Crossed products of operator algebras}

Let $\calA$ be an approximately unital (non-selfadjoint) operator algebra. We will always require that representations of $\calA$ be non-degenerate. An \textit{automorphism} on $\calA$ is a completely isometric isomorphism $\varphi:\calA \to \calA$. Note that if $\calA$ is unital, then any automorphism on $\calA$ is automatically unital. An \textit{operator algebra dynamical system} is a triple $(\calA,G, \alpha)$, where $\calA$ is an approximately unital operator algebra, $G$ is a locally compact group, and $\alpha: G \to \Aut(\calA)$ is a strongly continuous group homomorphism into the group $\Aut(\calA)$ of automorphisms on $\calA$. In \cite{katsoulisramsey}, crossed products for operator algebras are introduced. Katsoulis and Ramsey define these as operator subalgebras of a $C^{\ast}$-algebraic crossed product. To do this, they first define an $\alpha$\textit{-admissible} $C^{\ast}$\textit{-cover} of a dynamical system $(\calA,G,\alpha)$ to be any $C^{\ast}$-dynamical system $(\calC,G,\tilde{\alpha})$ with a completely isometric homomorphism $\rho: \calA \hookrightarrow \calC$ such that, for all $s \in G$, the diagram
\begin{center}
	\begin{tikzcd}
		 \calC \arrow{r}{\widetilde{\alpha}_s} & \calC \\
		 \calA \arrow{r}{\alpha_s} \arrow{u}{\rho} & \calA \arrow{u}{\rho}
	\end{tikzcd}
\end{center}
commutes. It is shown in \cite{katsoulisramsey} that, given a dynamical system $(\calA,G,\alpha)$, both the $C^{\ast}$-envelope and universal $C^{\ast}$-algebra of $\calA$ admit an $\alpha$-admissible $C^{\ast}$-cover, with action denoted by the symbol $\alpha$.
\begin{definition}
	Let $(\calA, G,\alpha)$ be an operator algebraic dynamical system, and let $(\calC,G,\alpha)$ be an $\alpha$-admissible $C^{\ast}$-cover with embedding $\rho: \calA \hookrightarrow \calC$.
	\begin{enumerate}
		\item The \textit{reduced crossed product} $\calA \rtimes^\lambda_\alpha G$ is the norm closure of $C_c(G,\rho(\calA))$ in $\calC \rtimes_{\alpha,\lambda} G$.
		\item The \textit{full crossed product relative to} $\calC$, denoted by $\calA \rtimes_{\calC,\alpha} G$, is the norm closure of $C_c(G,\rho(\calA))$ in $\calC \rtimes_{\alpha} G$.
		\item The \textit{full crossed product} $\calA \rtimes_{\alpha} G$ is the full crossed product relative to $C^*_{\max}(\calA)$ (see \cite{blecherlemerdy} for the definition of $C^*_{\max}(\calA)$).
	\end{enumerate}
\end{definition}

The reduced crossed product is independent of the choice of admissible $C^{\ast}$-cover: if $(\calC, G, \alpha)$ is any admissible $C^{\ast}$-cover of $(\calA, G, \alpha)$ with embedding $\rho: \calA \hookrightarrow \calC$, then the map
\begin{align*}
	\phi: C_c(G,\calA) \subseteq C^*_{\env}(\calA) \rtimes_\lambda G \to C_c(G,\rho(A)) \subseteq \calC \rtimes_\lambda G : f \mapsto \rho \circ f
\end{align*}
extends to a completely isometric isomorphism of operator algebraic crossed products \cite{katsoulisramsey}. Theorem \ref{KRexamplelocallycompactgroup} will show that the analogous result for full relative crossed products need not hold in general. As in the case of $C^{\ast}$-algebras, the full crossed product $\calA \rtimes_\alpha G$ of an operator algebra satisfies the following universal property: if $(\pi, u) : (\calA, G) \to B(\calH)$ is a covariant pair, in the sense that $\pi$ is a completely contractive homomorphism on $\calA$ and $u$ is a homomorphism on $G$ with $\pi(\alpha_s(a)) = u_s \pi(a) u_s^*$ for all $s \in G$ and $a \in \calA$, then there is a canonical completely contractive homomorphism
\begin{align*}
	\pi \rtimes u : \calA \rtimes_\alpha G \to B(\calH)
\end{align*}
extending $\pi$ and $u$. This map is called the \textit{integrated form} of $(\pi,u)$.

When the group action is trivial, the resulting crossed product structures become operator algebra tensor products.  We briefly recall the definition of these tensor products.  We refer the reader to \cite{blecherlemerdy} for more information on operator algebra tensor products.

\begin{definition}
Let $\calA \subseteq \bofh$ and $\calB \subseteq \calB(\calK)$ be approximately unital operator algebras. The \textit{minimal tensor product} of $\calA$ and $\calB$, denoted by $\calA \otimes_{\min} \calB$, is the completion of $\calA \otimes \calB$ with respect to the norm inherited from $\calB(\calH \otimes \calK)$.
\end{definition}

Note that matrix norms for $\calA \otimes_{\min} \calB$ are also inherited from $\calB(\calH \otimes \calK)$. The definition of the minimal tensor product does not depend on the choice of embeddings.

\begin{definition}
For approximately unital operator algebras $\calA$ and $\calB$, the \textit{maximal tensor product} of $\calA$ and $\calB$, denoted by $\calA \otimes_{\max} \calB$, is the completion of $\calA \otimes \calB$ with respect to the norm on $\calA \otimes \calB$ given by
$$\|x\|_{\max}=\sup \{ \|\pi \cdot \rho(x)\|_{\calB(\calH)} \},$$
where the supremum is taken over all Hilbert spaces $\calH$ and all completely contractive representations $\pi:\calA \to \bofh$ and $\rho:\calB \to \bofh$ with commuting ranges.
\end{definition}

We note that matrix norms for $\calA \otimes_{\max} \calB$ are defined similarly. In the supremum, the maps $\pi$ and $\rho$ can be assumed to be non-degenerate \cite[6.1.11]{blecherlemerdy}. Moreover, if $\calB$ is a $C^*$-algebra, then $\calA \otimes_{\max} \calB$ is completely isometrically contained in the $C^*$-algebraic tensor product $C_{\max}^*(\calA) \otimes_{\max} \calB$ \cite[6.1.9]{blecherlemerdy}.

\begin{example}
\label{example: opalg trivial actions}
	Let $\calA$ be an approximately unital operator algebra and let $G$ be a locally compact group. Let $\id: G \to \Aut(\calA)$ be the trivial action on $\calA$; that is, $\id_s = \id_\calA$ for all $s \in G$. We have natural isomorphisms
	\begin{align*}
		\calA \rtimes^\lambda_{\id} G &\simeq \calA \otimes_{\min} C^*_\lambda(G) : a\lambda_s \mapsto a \otimes \lambda_s, \text{ and } \\
		\calA \rtimes_{\id} G &\simeq \calA \otimes_{\max} C^*(G) : a u_s \mapsto a \otimes u_s\;.
	\end{align*}
	In the reduced case, we have the natural isomorphism $$C^*_{\env}(\calA) \rtimes_{\id, \lambda} G = C^*_{\env}(\calA) \otimes_{\min} C^*_\lambda(G),$$
	since the left regular representation is exactly the representation of the minimal tensor product. In the full case, we have the natural isomorphism
	$$C^*_{\max}(\calA) \rtimes_{\id} G = C^*_{\max}(\calA) \otimes_{\max} C^*(G),$$
	since the universal property for both algebras are identical. Restricting these isomorphisms to $C_c(G,\calA)$ yields the result. For more details on these isomorphisms, see \cite[Lemma 2.73 and Corollary 7.17]{danawilliams}.

\end{example}

\subsection{Operator system tensor products}

We briefly recall some facts about operator system tensor products here; more information can be found in \cite{KPTT}.

An \textit{operator system tensor product} $\tau$ is a map that sends a pair of operator systems $(\calS,\calT)$ to an operator system $\calS \otimes_{\tau} \calT$, such that
\begin{enumerate}
\item
If $X=(X_{ij}) \in M_n(\calS)_+$ and $Y=(Y_{k\ell}) \in M_m(\calT)_+$, then $X \otimes Y:=(X_{ij} \otimes Y_{k\ell}) \in M_{nm}(\calS \otimes_{\tau} \calT)_+$; and
\item
If $\varphi:\calS \to M_n$ and $\psi:\calT \to M_m$ are unital completely positive (ucp) maps, then $\varphi \otimes \psi:\calS \otimes_{\tau} \calT \to M_{nm}$ is ucp.
\end{enumerate}

An operator system tensor product $\tau$ is said to be \textit{symmetric} if, for every pair of operator systems $(\calS,\calT)$, the flip map $\calS \otimes \calT \to \calT \otimes \calS$ induces a complete order isomorphism $\calS \otimes_{\tau} \calT \to \calT \otimes_{\tau} \calS$.

We will be working with four main operator system tensor products:

\begin{definition}
The \textit{minimal tensor product} of $\calS$ and $\calT$, denoted by $\calS \otimes_{\min} \calT$, is defined such that $X \in M_n(\calS \otimes_{\min} \calT)$ is positive if and only if $(\varphi \otimes \psi)^{(n)}(X) \in M_{kmn}^+$ for every pair of ucp maps $\varphi:\calS \to M_k$ and $\psi:\calT \to M_m$.
\end{definition}

A fact that will be used throughout the paper is that the minimal tensor product is injective; i.e., whenever $\calS_1,\calS_2,\calT_1,\calT_2$ are operator systems with unital complete order embeddings $\iota:\calS_1 \subseteq \calS_2$ and $\kappa:\calT_1 \subseteq \calT_2$, the tensor product map $\iota \otimes \kappa:\calS_1 \otimes_{\min} \calT_1 \to \calS_2 \otimes_{\min} \calT_2$ is a complete order embedding \cite[Theorem 4.6]{KPTT}. In particular, if $\calS$ is an operator subsystem of $B(\calH)$ and $\calT$ is an operator subsystem of $B(\calK)$, then $\calS \otimes_{\min} \calT$ is completely order isomorphic to the image of $\calS \otimes \calT$ in $B(\calH \otimes \calK)$ \cite[Theorem 4.4]{KPTT}. 

For two linear maps $\varphi:\calS \to \bofh$ and $\psi:\calT \to \bofh$ with commuting ranges (i.e. $\varphi(s)\psi(t)=\psi(t)\varphi(s)$ for all $s \in \calS$ and $t \in \calT$), we let $\varphi \cdot \psi:\calS \otimes \calT \to \bofh$ be defined on the vector space tensor product by $\varphi \cdot \psi(s \otimes t)=\varphi(s)\psi(t)$.

\begin{definition}
The \textit{commuting tensor product} of $\calS$ and $\calT$, denoted by $\calS \otimes_c \calT$, is defined such that $X \in M_n(\calS \otimes_c \calT)$ is positive if and only if $(\varphi \cdot \psi)^{(n)}(X) \in M_n(\bofh)_+$ for every pair of ucp maps $\varphi:\calS \to \bofh$ and $\psi:\calT \to \bofh$ with commuting ranges.
\end{definition}

We note that $\calS \otimes_c \calT$ is completely order isomorphic to the inclusion $\calS \otimes \calT \subseteq C_u^*(\calS) \otimes_{\max} C_u^*(\calT)$ \cite[Theorem 6.4]{KPTT}.

\begin{definition}
The \textit{maximal tensor product} of $\calS$ and $\calT$, denoted by $\calS \otimes_{\max} \calT$, is defined such that $X \in M_n(\calS \otimes_{\max} \calT)$ is positive if and only if, for every $\varepsilon>0$, there are $S_{\varepsilon} \in M_{k(\varepsilon)}(\calS)_+$, $T_{\varepsilon} \in M_{m(\varepsilon)}(\calT)_+$ and a linear map $A_{\varepsilon}:\C^n \to \C^{k(\varepsilon)} \otimes \C^{m(\varepsilon)}$ such that
$$X+\varepsilon 1=A_{\varepsilon}^* (S_{\varepsilon} \otimes T_{\varepsilon}) A_{\varepsilon}.$$
\end{definition}

\begin{definition}
The \textit{essential tensor product} of $\calS$ and $\calT$, denoted by $\calS \otimes_{\ess} \calT$, is the operator system structure on $\calS \otimes \calT$ inherited from the inclusion $\calS \otimes \calT \subseteq C_{\env}^*(\calS) \otimes_{\max} C_{\env}^*(\calT)$.
\end{definition}

For any two operator system tensor products $\alpha$ and $\beta$, we write $\alpha \leq \beta$ if, for all operator systems $\calS$ and $\calT$, the identity map $\id:\calS \otimes_{\beta} \calT \to \calS \otimes_{\alpha} \calT$ is ucp. An operator system $\calS$ is said to be $(\alpha,\beta)$-\textit{nuclear} if for every operator system $\calT$, the identity map $\id:\calS \otimes_{\alpha} \calT \to \calS \otimes_{\beta} \calT$ is a complete order isomorphism. For example, every unital $C^*$-algebra is $(c,\max)$-nuclear \cite[Theorem 6.7]{KPTT}.

\subsection{Finite-dimensional operator system quotients and duals}

In general, the dual space of an operator system can always be made into a matrix-ordered $*$-vector space \cite[Lemma 4.2, Lemma 4.3]{choieffrosopsys} as follows. If $\calS$ is an operator system with Banach space dual $\calS^d$, and $f=(f_{ij}) \in M_n(\calS^d)$, then we define $f^*=(f_{ji}^*)$, where $f_{ij}^*(s):=\overline{f_{ij}(s^*)}$ for all $1 \leq i,j \leq n$ and $s \in \calS$.  We say that a self-adjoint element $f=(f_{ij}) \in M_n(\calS^d)$ is positive if the associated map $F:\calS \to M_n$ given by $F(s)=(f_{ij}(s))$ is completely positive.  With this structure, $\calS^d$ becomes a matrix-ordered $*$-vector space.  If $\calS$ is not finite-dimensional, then $\calS^d$ may not have an order unit, and hence may not be an operator system.  However, if $\calS$ is finite-dimensional, then $\calS^d$ is an operator system, and any faithful state on $\calS^d$ will be an order unit for $\calS^d$ \cite{choieffrosopsys}.

The theory of operator system quotients is rather new and not well understood. If $\varphi:\calS \to \calT$ is a surjective ucp map between operator systems, then we may endow the quotient vector space $\calS/\ker(\varphi)$ with an operator system structure \cite{KPTTquotients}. For $s \in \calS$, we write $\dot{s}$ to denote its image in $\calS/\ker(\varphi)$. We say that $\dot{X}=(\dot{X}_{ij}) \in M_n(\calS/\ker(\varphi))$ is positive if, for every $\varepsilon>0$, there is $Y_{\varepsilon} \in M_n(\calS)_+$ such that $\dot{Y}_{\varepsilon}=\dot{X}+\varepsilon \dot{I}_n$, where $I_n$ denotes the $n \times n$ identity matrix in $M_n(\calS)$. Note that whenever $\varphi:\calS \to \calT$ is a surjective ucp map, the induced map $\dot{\varphi}:\calS/\ker(\varphi) \to \calT$ is ucp \cite[Proposition 3.6]{KPTTquotients}. This leads to the following definition.

\begin{definition}
A surjective ucp map $\varphi:\calS \to \calT$ between operator systems is said to be a \textit{complete quotient map} if the induced map $\dot{\varphi}:\calS/\ker(\varphi) \to \calT$ is a complete order isomorphism.
\end{definition}

For our purposes, it will be helpful to translate between complete quotient maps and complete order embeddings, via the Banach space adjoint.  Recall that whenever $\varphi:\calS \to \calT$ is a ucp map between finite-dimensional operator systems, we may define a ucp map $\varphi^d:\calT^d \to \calS^d$ by $[\varphi^d(f)](s)=f(\varphi(s))$. Then a ucp map $\varphi:\calS \to \calT$ between finite-dimensional operator systems is a complete quotient map if and only if $\varphi^d:\calT^d \to \calS^d$ is a complete order embedding \cite[Proposition 1.8]{FP}.

\subsection{Notational conventions}
	In this paper all non-selfadjoint operator algebras will be assumed to be unital, and will be denoted by the script letters $\calA$ and $\calB$. Our groups will be denoted by the letter $G$ and are assumed to be discrete unless otherwise stated. Unless otherwise specified, all $C^{\ast}$-algebras will be assumed to be unital and will be denoted by the letters $\calC$ and $\calD$. Operator systems will generally be denoted by the letters $\calS$ and $\calT$. The order unit of an operator system $\calS$ will be denoted by the unit $1$ or $1_\calS$ if it needs to be specified, with the exception of the order unit of the matrix algebra $M_n(\C)$, which will be denoted by $I_n$.

\section{Reduced Crossed Products}

Let $G$ be a discrete group. We will assume that we are working with a set of generators $\frakg$ for $G$ such that $\frakg^{-1} = \frakg$ and $e \in \frakg$.

\begin{definition}
  If $\calS$ is an operator system, then $\text{Aut}(\calS)$ is the group of unital complete order isomorphisms $\varphi:\calS \to \calS$. An \textit{(operator system) dynamical system} is a $4$-tuple $(\calS,G,\frakg,\alpha)$, where $\calS$ is an operator system, $G$ is a group with generating set $\frakg$, and $\alpha: G \to \Aut(\calS)$ is a group homomorphism.

  Let $(\calS,G,\frakg,\alpha)$ be a dynamical system, and let $\rho:\calS \to \calC$ be a complete order embedding for which $C^*(\rho(\calS))=\calC$.  The $C^*$-cover $\calC$ is said to be $\alpha$-\textit{admissible} if there is a group action $\overline{\alpha}:G \to \Aut(\calC)$ for which the diagram
  \begin{center}
    \begin{tikzcd}
      \calC \arrow{r}{\cl{\alpha}_g} & \calC \\
      \calS \arrow{u}{\rho} \arrow{r}{\alpha_g} & \calS \arrow{u}{\rho}
    \end{tikzcd}
  \end{center}
  commutes for all $g \in G$. We denote such an $\alpha$-admissible $C^*$-cover by the triple $(\calC,\rho,\overline{\alpha})$.

  Given an $\alpha$-admissible $C^{\ast}$-cover $(\calC,\rho,\cl{\alpha})$ of a dynamical system $(\calS, G,\frakg,\alpha)$, the \textit{reduced crossed product relative to} $\frakg$ is defined as the operator subsystem of the reduced crossed product $C^*$-algebra $\calC \rtimes_{\overline{\alpha},\lambda} G$ given by
  \begin{align*}
    \calS \rtimes^{(\calC,\rho)}_{\alpha, \lambda} \frakg := \Span\{a \lambda_g : a \in \calS, g \in \frakg \} \subset \calC \rtimes_{\cl{\alpha}, \lambda} G\;.
  \end{align*}

Finally, given two operator system dynamical systems $(\calS,G,\frakg,\alpha)$ and $(\calT,G,\frakh,\beta)$, we say that a ucp map $\phi:\calS \to \calT$ is $G$\textit{-equivariant} if for every $g \in G$ and $s \in \calS$, we have $$\beta_g(\phi(s))=\phi(\alpha_g(s)).$$
\end{definition}

\begin{remark}\label{remark: dependence on generators, reduced}
  If $G$ is a discrete group with generating sets $\frakg$ and $\frakh$ and $\frakg \subset \frakh$ then for any $G$-action $\alpha: G \to \Aut(\calS)$ and any admissible $C^{\ast}$-cover $(\calC,\rho,\alpha)$ we have the complete order embedding
  \begin{align*}
    \calS \rtimes^{(\calC,\rho)}_{\alpha,\lambda} \frakg \subset \calS \rtimes^{(\calC,\rho)}_{\alpha,\lambda} \frakh,\;.
  \end{align*}
  given by the canonical inclusion. Thus, although different choices of generating sets will yield different crossed products in general, the choice of generators is not essential in the structure of the crossed product.
\end{remark}

\begin{remark}\label{remark: connection to operator algebra, reduced}
  Recall that if $\calA$ is a unital operator algebra with $C^*$-cover $(\calC,\rho)$ and $\varphi:\rho(\calA) \to B(\calH)$ is a linear map, then $\varphi$ is unital and completely contractive if and only if the map $\widetilde{\varphi}:\rho(\calA)+\rho(\calA)^* \to B(\calH)$ given by $$\widetilde{\varphi}(\rho(a)+\rho(b)^*)=\varphi(\rho(a))+\varphi(\rho(b))^*$$ 
  is ucp \cite{arveson69}. In particular, if $\alpha \in \text{Aut}(\rho(\calA))$, then it readily follows that $\widetilde{\alpha} \in \text{Aut}(\rho(\calA)+\rho(\calA)^*)$. Suppose that $G$ is a discrete group. For a group action $\alpha:G \to \text{Aut}(\calA)$ and an $\alpha$-admissible $C^*$-cover $(\calC,\rho,\alpha)$, there is an associated group action $\widetilde{\alpha}:G \to \text{Aut}(\rho(\calA)+\rho(\calA)^*)$ given by the assignment $g \mapsto \widetilde{\alpha}_g$. In fact, any $\alpha$-admissible $C^{\ast}$-cover $(\calC,\rho,\alpha)$ for $(\calA,G,\alpha)$ is also $\widetilde{\alpha}$-admissible for  $(\rho(\calA)+\rho(\calA)^*,G,\widetilde{\alpha},\frakg)$. Let $\widetilde{\rho}:\rho(\calA)+\rho(\calA)^* \to \calC$ be the canonical inclusion. Then for reduced crossed products, setting $\frakg=G$, we have the identity
  \begin{align*}
    (\rho(\calA) + \rho(\calA)^*) \rtimes^{(\calC,\widetilde{\rho})}_{\alpha,\lambda} G = (\calA \rtimes^{(\calC,\rho)}_{\alpha,\lambda} G) + (\calA \rtimes^{(\calC,\rho)}_{\alpha,\lambda} G)^* \subseteq \calC \rtimes_{\alpha,\lambda} G\;.
  \end{align*}
  This means that there is a bijective correspondence between unital completely positive maps on the reduced crossed product $(\rho(\calA) + \rho(\calA)^*) \rtimes^{(\calC,\widetilde{\rho})}_{\alpha,\lambda} G$ and unital completely contractive maps on $\calA \rtimes^{(\calC,\rho)}_{\alpha,\lambda} G$. In this way, any reduced crossed product of a unital operator algebra by a discrete group is contained completely isometrically in an associated operator system reduced crossed product.
\end{remark}

We would like an abstract notion of the reduced crossed product. Indeed, we shall show that the reduced crossed product is independent of its admissible $C^{\ast}$-cover. Until that fact is established, we will always make reference to the $C^*$-cover in question when discussing (relative) reduced crossed products.

\begin{example}
  Let $G$ be a group.  Consider the trivial action of $G$ on $\C$; i.e., $\alpha_g(1)=1$ for all $g \in G$. In this case, we simply recover the reduced group operator system corresponding to the generating set $\frakg$.  That is to say,
  \begin{align*}
    \C \rtimes^{C^*_\lambda(G)}_{\id,\lambda} \frakg = \calS_\lambda(\frakg),\;.
  \end{align*}
  where $\calS_{\lambda}(\frakg)=\text{span}\{ \lambda_g: g \in \frakg\} \subseteq C_{\lambda}^*(G)$ \cite{DiscGroup}.
\end{example}

\begin{proposition}
\label{proposition: reducedtrivialaction}
  Suppose that $(\calS, G,\frakg, id)$ is the trivial dynamical system; i.e., $\alpha_g = id_\calS$ for all $g \in G$. Then the map
  \begin{align*}
  \Psi:\calS \rtimes_{\id,\lambda} \frakg &\to \calS \otimes_{\min} \calS_{\lambda}(\frakg) \\
  s\alpha_g &\mapsto s \otimes \lambda_g
  \end{align*}
  is a complete order isomorphism.
\end{proposition}

\begin{proof}
  From the theory of $C^{\ast}$-algebras, there is an isomorphism $\Phi: C^*_{\env}(\calS) \rtimes_{id,\lambda} G \to C^*_{\env}(\calS) \otimes_{\min} C^*_{\lambda}(G)$ which sends generators to generators \cite[Lemma 2.73]{danawilliams}. This restricts to an isomorphism $\Psi: \calS \rtimes_{id,\lambda} \frakg \to \Span\{a \otimes \lambda_g: a \in \calS, g \in \frakg\}$ which sends generators to generators. By \cite[Corollary 4.10]{KPTT}, the latter operator system is precisely $\calS \otimes_{\min} \calS_{\lambda}(\frakg)$.
\end{proof}

Let $\calC$ be a $C^*$-algebra and $\calS$ be an operator system contained in $\calC$.  We say that $\calS$ \textit{contains enough unitaries} in $\calC$ if the set of elements in $\calS$ which are unitary in $\calC$ generate $\calC$ as a $C^*$-algebra. This property of operator systems was first considered in \cite{KPTTquotients}. A result of Kavruk \cite[Proposition 5.6]{kavruk14} states that if $\calS \subset \calC$ is an operator subsystem of a $C^{\ast}$-cover $\calC$ for which $\calS$ contains enough unitaries, then $\calC$ is the $C^{\ast}$-envelope of $\calS$. In particular, $C^*_{\env}(\C \rtimes^{C^*_\lambda(G)}_{\id,\lambda} \frakg) = C^*_\lambda(G)$.

Before working more with $C^*$-envelopes corresponding to dynamical systems, we first show that for any dynamical system, the group action can be extended to a group action on the $C^*$-envelope.

\begin{proposition}
  Suppose that $(\calS,G,\frakg, \alpha)$ is a dynamical system. Suppose that $(C^*_{\env}(\calS),\iota)$ is the $C^{\ast}$-envelope of $\calS$, where $\iota: \calS \hookrightarrow C^*_{\env}(\calS)$ is the canonical complete order embedding. Then there exists a $G$-action $\overline{\alpha}$ on $C^*_{\env}(\calS)$ which makes $(C^*_{\env}(\calS), \iota, \overline{\alpha})$ an $\alpha$-admissible $C^{\ast}$-cover of $\calS$.
\end{proposition}

\begin{proof}
  Let $g \in G$. Since $\iota \circ \alpha_g$ is a complete order embedding of $\calS$ into $C^*_{\env}(\calS)$, by the universal property of $C^{\ast}$-envelopes, there is a unique surjective  $*$-homomorphism $\overline{\alpha}_g: C^*_{\env}(\calS) \to C^*_{\env}(\calS)$ for which the diagram
  \begin{center}
    \begin{tikzcd}
      C^*_{\env}(\calS) \arrow{r}{\alpha_g} & C^*_{\env}(\calS) \\
      \calS \arrow{u}{\iota} \arrow{r}{\alpha_g} & \calS \arrow{u}{\iota}
    \end{tikzcd}
  \end{center}
  commutes. For every $g,h \in G$, the map $\overline{\alpha}_{gh} \circ (\overline{\alpha}_g \circ \overline{\alpha}_h)^{-1}$ restricts to the identity on $\calS$.  By uniqueness, we must have $\overline{\alpha}_{gh} \circ (\overline{\alpha}_g \circ \overline{\alpha}_h)^{-1}=\id_{C_{\env}^*(\calS)}$, so that $\overline{\alpha}_{gh}=\overline{\alpha}_g \circ \overline{\alpha}_h$. Evidently we have $\overline{\alpha}_e=\id_{C_{\env}^*(\calS)}$. In particular, $\overline{\alpha}_g \circ \overline{\alpha}_{g^{-1}}=\overline{\alpha}_{g^{-1}} \circ \overline{\alpha}_g=\id_{C_{\env}^*(\calS)}$, so each $\overline{\alpha}_g$ is an automorphism. Then $\overline{\alpha}: G \acts C^*_{\env}(\calS)$ is a group action which is admissible with respect to the dynamical system.
\end{proof}

\begin{definition}
	Let $(\calS,G,\frakg,\alpha)$ be a dynamical system. Define the \emph{reduced crossed product} $\calS \rtimes_{\alpha,\lambda} \frakg$ to be the reduced crossed product $\calS \rtimes_{\alpha,\lambda}^{(C^*_{env}(\calS),\iota)} \frakg$ relative to $(C^*_{env}(\calS),\iota)$.
\end{definition}

\subsection{The $C^{\ast}$-envelope of a reduced crossed product}

The goal of this section is to prove the identity
\begin{align*}
  C^*_{\env}(\calS \rtimes_{\alpha,\lambda} \frakg) = C^*_{\env}(\calS) \rtimes_{\alpha,\lambda} G
\end{align*}
for any dynamical system $(\calS, G, \frakg, \alpha)$.

The next lemma contains some useful facts relating to hyperrigidity.

\begin{lemma}
\label{lemma: uep properties}
  Let $\calS$ be an operator system.
  \begin{enumerate}
    \item If $\alpha$ is an automorphism on $\calS$ and $\pi: \calS \to B(\calH)$ is a representation with the unique extension property, then $\pi \circ \alpha$ has the unique extension property.
    \item If $\pi_i: \calS \to B(\calH_i)$ is a maximal representation for each $i \in I$, then $\bigoplus_i \pi_i$ is also maximal.
  \end{enumerate}
\end{lemma}

\begin{proof}
The fact that (1) holds is by the definition of the unique extension property.  The proof of (2) is due to Arveson \cite[Proposition 4.4]{arveson11}.
\end{proof}

\begin{definition}[(Arveson, \cite{arveson11})]
  Let $\calS$ be an operator system in a $C^{\ast}$-algebra $\calC$. We say that $\calS$ is \textit{hyperrigid} in $\calC$ if whenever $\pi: \calC \to B(\calH)$ is a $*$-representation, the map $\pi{|_\calS}$ satisfies the unique extension property.
\end{definition}

The following theorem is helpful for our purposes.

\begin{theorem}\label{theorem: hyperrigid is envelope}
  If $\calS$ is an operator system in a unital $C^*$-algebra $\calC$ and $\calS$ is hyperrigid in $\calC$, then $\calC=C_{\env}^*(\calS)$.
\end{theorem}

\begin{proof}
  If $\calS \subset \calC$ is hyperrigid, then any faithful representation $\pi$ of $\calC$ is such that $\pi_{\calS}$ has the unique extension property. Thus, $\pi_{\calS}$ is maximal on $\calS$. By a theorem of Dritschel and McCollough \cite[Theorem 1.1]{dritschelmccullough}, since maximal representations generate $C^{\ast}$-envelopes, $C^*(\pi_{|\calS})$ is isomorphic to $C_{\env}^*(\calS)$.  Evidently $\calC \simeq C^*(\pi_{|\calS})$, so this proves that $\calC$ is the $C^{\ast}$-envelope of $\calS$.
\end{proof}

Recall that an operator subsystem $\calS$ of a unital $C^*$-algebra $\calC$ \textit{contains enough unitaries} if the the set of elements in $\calS$ that are unitary in $\calC$ generates $\calC$ as a $C^*$-algebra. The next result is already known. For convenience we include the proof.

\begin{lemma}
\label{containsenoughunitariesimplieshyperrigid}
  Suppose that $\calC$ is a unital $C^*$-algebra and that $\calS \subseteq \calC$ is an operator system which contains enough unitaries in $\calC$. Then the operator system $\calS$ is hyperrigid.
\end{lemma}

\begin{proof}
  Let $\pi:\calC \to B(\calH)$ be a $*$-representation, and suppose that the ucp map $\pi|_{\calS}$ is not maximal. Let $\psi: \calS \to B(\calK)$ be a maximal dilation of $\pi|_{\calS}$. Since $\psi$ is maximal, $\psi$ induces a $*$-homomorphism on $C^*(\calS) = \calC$, which we will also denote by $\psi$. For $s \in \calS$, we may decompose the operator $\psi(s)$ with respect to $\calK=\calH \oplus \calH^{\perp}$ as 
  \begin{align*}
  	\psi(s) = \left[\begin{array}{cc} \pi(s) & a_s \\ b_s & \chi(s) \end{array} \right]
  \end{align*}
  for operators $a_s \in B(\calH^{\perp},\calH)$, $b_s \in B(\calH,\calH^{\perp})$ and $\chi(s) \in B(\calH^{\perp})$. Let $u \in \calS$ be unitary in $\calC$. Since $\psi(u)$ must also be unitary, we know that the (1,1)-corner of the operators $\psi(u)\psi(u^*)$ and $\psi(u^*)\psi(u)$ must be the identity. A calculation shows that $(\psi(u)\psi(u^*))_{1,1} = I_{\calH} + a_{u}a_{u}^*$ and $(\psi(u^*)\psi(u))_{1,1} = I_{\calH} + b_{u}^*b_{u}$. Thus, both $a_{u}$ and $b_{u}$ must be zero, so that $\psi(u)=\pi(u) \oplus \chi(u)$. Therefore, if $u,v \in \calS$ are unitaries in $\calC$, then using the fact that $\psi$ is a $*$-homomorphism,
  $$\left[\begin{array}{cc} \pi(u)\pi(v) & 0 \\ 0 & \chi(u)\chi(v) \end{array}\right]=\psi(u)\psi(v)=\psi(uv)=\left[\begin{array}{cc} \pi(uv) & a_{uv} \\ b_{uv} & \chi(uv)\end{array}\right].$$
  Hence, $a_{uv}$ and $b_{uv}=0$. It easily follows that for any elements $u_1,...,u_n \in \calS$ that are unitary in $\calC$, we have
  $$\psi(u_1 \cdots u_n)=\left[\begin{array}{cc} \pi(u_1 \cdots u_n) & 0 \\ 0 & \chi(u_1 \cdots u_n) \end{array}\right].$$
  Since $\calC$ is generated by unitaries in $\calS$, $\calC$ is the span of elements of the form $u_1 \cdots u_n$ for elements $u_1,...,u_n$ of $\calS$ that are unitary in $\calC$. It follows that $\psi$ decomposes as $\pi \oplus \chi$ for some ucp map $\chi$ on $\calC$. Restricting to $\calS$, this proves that $\pi|_{\calS}$ is maximal, which is a contradiction. 
\end{proof}

This gives an alternate proof of Kavruk's result on $C^*$-envelopes \cite[Proposition 5.6]{kavruk14}.

\begin{corollary}
  Suppose that $\calC$ is a unital $C^*$-algebra and $\calS \subset \calC$ is an operator system that contains enough unitaries in $\calC$. Then $\calC$ is the $C^{\ast}$-envelope of $\calS$.
\end{corollary}

\begin{proof}
  By Lemma \ref{containsenoughunitariesimplieshyperrigid}, $\calS$ is hyperrigid. By Theorem~\ref{theorem: hyperrigid is envelope}, $\calC=C_{\env}^*(\calS)$.
\end{proof}

In order to show that $C_{\env}^*(\calS \rtimes_{\alpha,\lambda} \frakg)=C_{\env}^*(\calS) \rtimes_{\alpha} G$, we require the following lemma.

\begin{lemma}\label{lemma: maximals lift to reduced crossed product}
Suppose that $(\calS,G,\frakg,\alpha)$ is a dynamical system. Let $\pi:C_{\env}^*(\calS) \hookrightarrow B(\calH)$ be a faithful representation that is maximal on $\calS$. Let $(\overline{\pi},\lambda_{\calH})$ be the covariant extension of $\pi$ to $\calH \otimes \ell^2(G)$. Then the integrated form $\overline{\pi} \rtimes \lambda_{\calH}:C_{\env}^*(\calS) \rtimes_{\alpha,\lambda} G \to B(\calH \otimes \ell^2(G))$ is maximal on $\calS \rtimes_{\alpha,\lambda} \frakg$.
\end{lemma}

\begin{proof}
Since $\overline{\pi} = \bigoplus_{g \in G} \pi \circ \alpha_g$, Lemma \ref{lemma: uep properties} shows that $\overline{\pi}$ has the unique extension property on $\calS$. We claim that $ \overline{\pi} \rtimes \lambda_\calH$ has the unique extension property on $\calS \rtimes_{\alpha,\lambda} \frakg$. If this were true, then $\overline{\pi} \rtimes \lambda_{\calH}$ would be maximal on $\calS \rtimes_{\alpha,\lambda} \frakg$. Thus, it remains to show that $\overline{\pi} \rtimes \lambda_{\calH}$ has the unique extension property. Suppose that $\rho: C^*_{\env}(\calS) \rtimes_{\alpha,\lambda} G \to B(\calH)$ is a ucp extension of $ \overline{\pi}\rtimes \lambda_\calH|_{\calS \rtimes_{\alpha,\lambda} \frakg}$. We observe that, by maximality of $ \overline{\pi}$, we must have $\rho|_{C^*_{\env}(\calS)} =  \overline{\pi} \rtimes \lambda_\calH|_{C^*_{\env}(\calS)}=\pi$. Recall that $\calS_\lambda(\frakg) = \Span\{\lambda_g: g \in \frakg\}$. By Lemma \ref{containsenoughunitariesimplieshyperrigid}, $\calS_{\lambda}(\frakg)$ is hyperrigid, since it contains enough unitaries in its $C^*$-envelope. Since $\calS_\lambda(\frakg)$ is hypperrigid, we get the identity $\rho|_{C^*(\calS_\lambda(\frakg))} =  \overline{\pi} \rtimes \lambda_\calH|_{C^*(\calS_\lambda(\frakg))}$. Thus, $C^*(\calS_{\lambda}(\frakg))$ is in the multiplicative domain of $\rho$. Now, let $g \in G$. Since $\lambda_g \in C^*(\calS_\lambda(\frakg))$, for $a \in C^*_{\env}(\calS)$ we obtain the identity
  \begin{align*}
    \rho(a \lambda_g) = \rho(a)\rho(\lambda_g) =  (\overline{\pi} \rtimes \lambda_\calH(a))(\overline{\pi} \rtimes \lambda_\calH(\lambda_g)) =  \overline{\pi}\rtimes \lambda_\calH(a \lambda_g)\;.
  \end{align*}
  This proves that $\rho =  \overline{\pi} \rtimes \lambda_\calH$, so that $\overline{\pi} \rtimes \lambda_{\calH}$ has the unique extension property.
\end{proof}

We are now in a position to prove the desired result on $C^*$-envelopes of reduced crossed products.

\begin{theorem}\label{theorem: C*-envelope of reduced}
  Suppose that $(\calS,G,\frakg, \alpha)$ is a dynamical system. Then there is a canonical isomorphism
  \begin{align*}
    C^*_{\env}(\calS \rtimes_{\alpha,\lambda} \frakg) \simeq C^*_{\env}(\calS) \rtimes_{\alpha,\lambda} G\;.
  \end{align*}
\end{theorem}

\begin{proof}
    Let $\pi:C_{\env}^*(\calS) \hookrightarrow B(\calH)$ be a maximal representation, and let $\overline{\pi} \rtimes \lambda_{\calH}$ be the associated integrated form of the covariant extension $(\overline{\pi},\lambda_{\calH})$ of $\pi$. By Lemma \ref{lemma: maximals lift to reduced crossed product}, $\overline{\pi} \rtimes \lambda_{\calH}$ is maximal on $\calS \rtimes_{\alpha,\lambda} \frakg$. Thus, the $C^*$-algebra generated by $(\overline{\pi} \rtimes \lambda_{\calH})(\calS \rtimes_{\alpha,\lambda} \frakg)$ is the $C^*$-envelope of $\calS \rtimes_{\alpha,\lambda} \frakg$.  Thus, we obtain the isomorphism
  \begin{align*}
     \overline{\pi} \rtimes \lambda_\calH : C^*_{\env}(\calS) \rtimes_{\alpha,\lambda} G \to C^*_{\env}(\calS \rtimes_{\alpha,\lambda} \frakg)\;.
  \end{align*}
	which completes the proof.
\end{proof}

Hyperrigidity is also preserved by the reduced crossed product.

\begin{corollary}\label{cor: hyperrigid reduced}
  Suppose that $(\calS,G,\frakg,\alpha)$ is a dynamical system. If $\calS$ is hyperrigid in $C_{\env}^*(\calS)$, then $\calS \rtimes_{\alpha,\lambda} \frakg$ is hyperrigid in $C_{\env}^*(\calS \rtimes_{\alpha,\lambda} \frakg)$.
\end{corollary}

\begin{proof}
By Theorem \ref{theorem: C*-envelope of reduced}, we have $C_{\env}^*(\calS \rtimes_{\alpha,\lambda} \frakg)=C_{\env}^*(\calS) \rtimes_{\alpha,\lambda} G$. Suppose that $\rho:C_{\env}^*(\calS) \rtimes_{\alpha,\lambda} G \to B(\calH)$ is a unital $*$-homomorphism and let $\phi:C_{\env}^*(\calS) \rtimes_{\alpha,\lambda} G \to B(\calH)$ be a ucp extension of $\rho_{|\calS \rtimes_{\alpha,\lambda} \frakg}$. Since $\calS$ is hyperrigid in $C_{\env}^*(\calS)$, $\phi_{|C_{\env}^*(\calS)}$ satisfies the unique extension property on $\calS$. On the other hand, $\calS_{\lambda}(\frakg)$ is hyperrigid in $C_{\lambda}^*(G)$ by Lemma \ref{containsenoughunitariesimplieshyperrigid}. Thus, $\phi$ agrees with $\rho$ when restricted to the copy of $C_{\lambda}^*(G)$ in $C_{\env}^*(\calS) \rtimes_{\alpha,\lambda} G$. In particular, $C_{\env}^*(\calS)$ and the copy of $C_{\lambda}^*(G)$ are contained in the multiplicative domain of $\phi$.  As these two algebras generate $C_{\env}^*(\calS) \rtimes_{\alpha,\lambda} G$ as a $C^*$-algebra, it follows that $\rho=\phi$.  Therefore, $\rho$ is maximal on $\calS \rtimes_{\alpha,\lambda} \frakg$.  Since $\rho$ was an arbitrary representation of $C_{\env}^*(\calS) \rtimes_{\alpha,\lambda} G$, it follows that $\calS \rtimes_{\alpha,\lambda} \frakg$ is hyperrigid in $C_{\env}^*(\calS \rtimes_{\alpha,\lambda} \frakg)$.
\end{proof}

\begin{example}
\label{rotationalgebras}
  Consider the commutative $C^*$-algebra $C(\T)$ with generator $u: \T \to \C$ given by $u(z)=z$. Fix $\theta \in [0,1]$ and define the action $\alpha:\Z \acts C(\T)$ by the automorphism
  \begin{align*}
    \alpha : u \mapsto e^{2 \pi i \theta} u\;.
  \end{align*}
  Define $\calS_\T := \Span\{1, u ,u^*\}$. Observe that $\alpha$ restricts to an action on $\calS_\T$. Set $\frakg = \{1,0,-1\} \subset \Z$. The crossed product $\calS_\T \rtimes_{\alpha,\lambda} \frakg$ has $C^{\ast}$-envelope $C(\T) \rtimes_\alpha \Z$. Therefore, all rotation algebras are $C^{\ast}$-envelopes of finite-dimesional operator systems.
\end{example}

\begin{example}
  We consider a generalization of Example \ref{rotationalgebras}. Let $n \geq 1$, and let $U(n)$ act on the Cuntz algebra $\calO_n$ via the mapping
  \begin{align*}
    \alpha_g : s_i \mapsto \sum_{j=1}^n g_{ji} s_j\;
  \end{align*}
  where $s_1,\ldots, s_n$ are the isometries generating $\calO_n$ and $g=(g_{ij})$ is the matrix representation of the element $g$ of $U(n)$ with respect to the canonical basis. For a subgroup $G$ of $U(n)$, we say that an action $G \acts \calO_n$ is quasi-free if $G$ acts by $\alpha$ (see \cite{quasifree}). Let $\frakg \subset U(n)$ be a finite symmetric subset containing the identity. Let $G = \Angle{\frakg}$. Set $\calS_n = \Span\{ s_1,\ldots, s_n , 1, s_1^*,\ldots, s_n^*\}$. If $G \acts_\alpha \calO_n$ is a quasi-free action, then $G$ restricts to an action on $\calS_n$. The system $\calS_n \rtimes_{\alpha,\lambda} \frakg$ has $C^{\ast}$-envelope $\calO_n \rtimes_{\alpha,\lambda} G$.
\end{example}

\subsection{An abstract characterization of reduced crossed products}
We now move towards showing that the reduced crossed product does not depend on the choice of $C^*$-cover for the operator system. Recall the following characterization of positivity for reduced $C^*$-algebraic crossed products (see \cite[Corollary 4.1.6]{BrownAndOzawa}):
\begin{proposition}
  Suppose that $(\calC,G,\alpha)$ is a $C^{\ast}$-dynamical system. An element $x = \sum_{g \in G} a_g \lambda_g \in C_c(G,\calC)$ is positive if and only if for any finite set $\{g_1,\ldots, g_n\} \subset G$, the matrix
  \begin{align*}
    \left[\begin{array}{ccc} & & \\ & \alpha_{g_i}^{-1}(a_{g_ig_j^{-1}}) & \\ & & \end{array} \right]_{i,j=1}^n
  \end{align*}
  is positive in $M_n(\calC)$.
\end{proposition}

The following is a well known result. For a proof, see \cite[Lemma 7.16]{danawilliams}.
\begin{proposition}
\label{proposition: matrix algebras over cstar crossed product}
  Let $(\calC,G,\alpha)$ be a $C^{\ast}$-dynamical system and let $n \geq 1$. We have the isomorphism
  \begin{align*}
    M_n (\calC \rtimes_{\alpha,\lambda} G ) \simeq M_n(\calC) \rtimes_{\alpha^{(n)},\lambda} G\;.
  \end{align*}
\end{proposition}

The next corollary immediately follows from Proposition \ref{proposition: matrix algebras over cstar crossed product}.

\begin{corollary}
  Let $(\calS,G,\frakg,\alpha)$ be a dynamical system, and let $(A,\rho)$ be an $\alpha$-admissible $C^{\ast}$-cover. For $n \geq1$, we have a complete order isomorphism
  \begin{align*}
    M_n(\calS \rtimes_{\alpha,\lambda}^{(\calC,\rho)} \frakg) \simeq M_n(\calS) \rtimes_{\alpha^{(n)}, \lambda}^{(M_n(\calC),\rho^{(n)})} \frakg\;.
  \end{align*}
\end{corollary}

Therefore, we have the following characterization of reduced crossed products.
\begin{proposition}\label{prop: abstract char of pos}
  Let $(\calS,G,\frakg,\alpha)$ be a dynamical system and let $(\calC,\rho)$ be an $\alpha$-admissible $C^{\ast}$-cover. For $n \geq1$, the positive cones $C_n := M_n(\calS \rtimes_{\alpha,\lambda}^{(\calC,\rho)} \frakg)^+$ are given by the following rule: an element $x = \sum_{g \in \frakg} x_g \lambda_g \in M_n(\calS \rtimes_{\alpha,\lambda}^{(\calC,\rho)} \frakg)$ is in $C_n$ if and only if, for every finite subset $F$ of $G$, the matrix
  \begin{align*}
    \left[\begin{array}{ccc} & & \\ & \alpha^{(n)}_{g^{-1}}(\rho(x_{gh^{-1}})) & \\ & & \end{array} \right]_{g,h \in F}
  \end{align*}
  is positive in $M_{F}(M_n(\rho(\calS)))$.
\end{proposition}

We now prove that the reduced crossed product is independent of the $C^{\ast}$-cover. The proof is essentially the same as the proof of the analogous result for the operator algebras \cite[Lemma 3.11]{katsoulisramsey}.

\begin{lemma}
\label{lemma: automorphism that preserves embedding preserves Shilov ideal}
  Suppose that $\rho: \calS \hookrightarrow \calC$ is a complete order embedding of an operator system $\calS$ into a $C^{\ast}$-cover $\calC$. Let $\calJ_\calS$ be the Shilov ideal of $\calS$ in $\calC$. If $\alpha: \calC \to \calC$ is an automorphism such that $\alpha(\rho(\calS))=\rho(\calS)$, then $\alpha(\calJ_\calS)=\calJ_{\calS}$.
\end{lemma}

\begin{proof}
  Let $n \geq 1$ and $x \in M_n(\calS)$. A calculation shows that
  \begin{align*}
    \|\rho(x) + M_n(\alpha(\calJ_\calS)) \| = \|(\alpha^{-1})^{(n)}(\rho(x)) + M_n(\calJ_\calS)\| = \|(\alpha^{-1})^{(n)}(\rho(x))\| = \|x\|\;,
  \end{align*}
  by definition of $\calJ_{\calS}$.  Therefore, $\calS \to \calC/\alpha(\calJ_\calS) : x \mapsto \rho(x) + \alpha(\calJ_\calS)$ is a complete order isometry. Thus, $\alpha(\calJ_{\calS})$ is a boundary ideal for $\calS$ in $(\calC,\rho)$. Since the Shilov ideal is maximal amongst boundary ideals, $\alpha(\calJ_\calS) \subset \calJ_\calS$. Since $\alpha$ is an automorphism, applying the same argument for $\alpha^{-1}$ shows that $\calJ_{\calS} \subseteq \alpha(\calJ_{\calS})$. Thus, $\alpha(\calJ_\calS) = \calJ_\calS$.
\end{proof}

\begin{lemma}\label{lemma: alpha-admissible passes to quotients}
  Suppose that $(\calS,G,\frakg,\alpha)$ is a dynamical system and suppose that $(\calC,\rho,\alpha)$ is an $\alpha$-admissible $C^{\ast}$-cover. Then the $G$-action $\alpha$ induces a $G$-action on $\calC/\calJ_\calS$ via $g \mapsto \dot{\alpha}_g$, where $\dot{\alpha}_g(x+\calJ_{\calS})=\alpha_g(x)+\calJ_{\calS}$. Moreover, if $q:\calC \to \calC/\calJ_{\calS}$ is the canonical quotient map, then $(\calC/\calJ_\calS,q\circ \rho,\dot{\alpha})$ is an $\alpha$-admissible $C^{\ast}$-cover for $(\calS,G,\frakg,\alpha)$.
\end{lemma}

\begin{proof}
  Let $g \in G$. Since $(\calC,\rho,\alpha)$ is an $\alpha$-admissible $C^*$-cover, $\alpha_g(\rho(\calS))=\rho(\calS)$. By Lemma \ref{lemma: automorphism that preserves embedding preserves Shilov ideal}, $\alpha_g(\calJ_\calS) = \calJ_\calS$. Hence, the map $\dot{\alpha}_g$ as defined above is a well-defined unital $*$-homomorphism. It is easy to check that the assignment $g \mapsto \dot{\alpha}_g$ induces a group action $\dot{\alpha}$ of $G$ on $\calC/\calJ_\calS$.

  To see that $(\calC/\calJ_\calS,q\circ \rho,\dot{\alpha})$ is an $\alpha$-admissible $C^{\ast}$-cover for $(\calS,G,\frakg,\alpha)$, let $x \in \calS$ and let $g \in G$. We have the identity
  \begin{align*}
    \dot{\alpha}_g(q\circ\rho(x)) &= \dot{\alpha}_g(\rho(x) + \calJ_\calS) = (\alpha_g \circ \rho)(x) + \calJ_\calS \\
    &= \rho(\alpha_g(x)) + \calJ_\calS = q \circ \rho (\alpha_g(x)),
  \end{align*}
  thus proving that $(\calC/\calJ_\calS,q\circ \rho,\dot{\alpha})$ is $\alpha$-admissible.
\end{proof}

\begin{theorem}
  Suppose that $(\calS,G,\frakg,\alpha)$ is a dynamical system and suppose that $(\calC,\rho,\alpha)$ is an $\alpha$-admissible $C^{\ast}$-cover. If $\calJ_\calS$ is the Shilov boundary of $\calS$ in $\calC$, then
  \begin{align*}
    \calS \rtimes_{\alpha,\lambda}^{(\calC/\calJ_\calS,q \circ \rho)} \frakg \simeq \calS \rtimes_{\alpha,\lambda}^{(\calC,\rho)} \frakg
  \end{align*}
  canonically. In particular, the reduced crossed product does not depend on the choice of $C^*$-cover.
\end{theorem}

\begin{proof}
  By Lemma~\ref{lemma: alpha-admissible passes to quotients}, the reduced crossed product $\calS \rtimes_{\alpha,\lambda}^{(\calC/\calJ_\calS,q \circ \rho)} \frakg$ is well-defined. It remains to prove that the map
  \begin{align*}
    \Phi:\calS \rtimes_{\alpha,\lambda}^{(\calC,\rho)} \frakg &\to \calS \rtimes_{\alpha,\lambda}^{(\calC/\calJ_\calS,q \circ \rho)} \frakg \\
    x\lambda_g &\mapsto (x+\calJ_\calS)\lambda_g
  \end{align*}
  is a complete order isomorphism. This map is ucp since it arises as the restriction of a *-homomorphism
  \begin{align*}
  		\calC \rtimes_{\alpha,\lambda} G &\to \calC/\calJ_\calS \rtimes_{\dot{\alpha},\lambda} G \\
  		x\lambda_s &\mapsto (x+\calJ_\calS)\lambda_s\;.
  \end{align*}
  Conversely, for $X \in M_n(\calS \rtimes_{\alpha,\lambda}^{(\calC,\rho)} \frakg)$, suppose that $\Phi(X) \in M_n(\calS \rtimes_{\alpha,\lambda}^{(\calC/\calJ_{\calS},q \circ \rho)} \frakg)$ is positive. By Proposition~\ref{prop: abstract char of pos}, $X$ is positive if and only if for every finite subset $F$ of $G$, the matrix
  \begin{align*}
  	\left[\begin{array}{ccc} & & \\ & \dot{\alpha}_{g^{-1}}^{(n)}(q \circ \rho(X_{gh^{-1}})) & \\ & & \end{array} \right]_{g,h \in F}
  \end{align*}
  is positive in $M_F(M_n(q \circ \rho(\calS)))$. $\dot{\alpha}_g \circ (q \circ \rho)=(q \circ \rho) \circ \alpha_g$ and $q \circ \rho:\calS \to \calC/\calJ_{\calS}$ is a complete order embedding, we see that the matrix
  \begin{align*}
  	\left[\begin{array}{ccc} & & \\ & \alpha_{g^{-1}}^{(n)}(X_{gh^{-1}}) & \\ & & \end{array} \right]_{g,h\in F}
  \end{align*}
  is positive in $M_F(M_n(\calS))$. Since $\alpha_g \circ \rho=\rho \circ \alpha_g$, the matrix
  \begin{align*}
    	\left[\begin{array}{ccc} & & \\ & \alpha_{g^{-1}}^{(n)}(\rho(X_{gh^{-1}})) & \\ & & \end{array} \right]_{g,h\in F}
    \end{align*}
    is positive in $M_F(M_n(\rho(\calS))$. Applying Proposition~\ref{prop: abstract char of pos} again, we see that $X$ is positive in $M_n(\calS \rtimes_{\alpha,\lambda}^{(\calC,\rho)} \frakg)$, establishing the complete order isomorphism.

\end{proof}

We close this section with a short discussion on $G$-equivariant ucp maps. First, we have the following result (\cite[Exercise 4.1.4]{BrownAndOzawa}):
\begin{proposition}\label{prop: covariance under ucp}
  Suppose that $(\calS,G,\frakg,\alpha)$ and $(\calT,G,\frakg,\beta)$ are dynamical systems and suppose that $\phi: \calS \to \calT$ is a $G$-equivariant ucp map. The map
  \begin{align*}
    \widetilde{\phi}: \calS \rtimes_{\alpha,\lambda} \frakg &\to \calT \rtimes_{\beta,\lambda} \frakg : a\lambda_g \mapsto \phi(a) \lambda_g
  \end{align*}
  is ucp. If the map $\phi$ is a complete order embedding, then the map $\tilde{\phi}$ is also a complete order embedding.
\end{proposition}

\begin{proof}
  It is clear that $\widetilde{\phi}$ is unital. Since the amplifications $\phi^{(n)} : M_n(\calS) \to M_n(\calT)$ are $G$-equivariant, it suffices to show that $\widetilde{\phi}$ is positive. If $x = \sum_{g \in \frakg} x_g \lambda_g$ is positive in $\calS \rtimes_{\alpha,\lambda} \frakg$, then for each finite $F \subset G$ the matrix
  \begin{align*}
    P:= \left[\begin{array}{ccc} & & \\ & \alpha_{g^{-1}}(x_{gh^{-1}}) & \\ & & \end{array} \right]_{g,h \in F}\;.
  \end{align*}
  is positive in $M_{F}(\calS)$. Since $\phi$ is ucp, $\phi^{(F)}(P) \geq 0$ in $\calT$. By $G$-equivariance, this means that the matrix
  \begin{align*}
    \left[\begin{array}{ccc} & & \\ & \beta_{g^{-1}}(\phi(x_{gh^{-1}})) & \\ & & \end{array} \right]_{g,h \in F}
  \end{align*}
  is positive in $M_{F}(\calT)$. This occurs if and only if the element $\sum_{g \in \frakg} \phi(x_g)\lambda_g$ is positive in $\calT \rtimes_{\beta, \lambda} \frakg$. A similar argument shows that $\widetilde{\varphi}$ is a complete order embedding whenever $\varphi$ is a complete order embedding.
\end{proof}

In the case of $C^{\ast}$-algebras, a $G$-equivariant quotient map between two $C^{\ast}$-algebras produces a quotient map on the reduced crossed product. This fails in the case of operator systems. For example, let $\frakz := \{1,0,-1\} \subset \Z$. Let $E_{00},E_{01},E_{10}, E_{11}$ enumerate the canonical system of matrix units for $M_2$. Define
	\begin{align*}
		\phi: M_2 \to \calS(\frakz) : E_{ij} \mapsto \frac{1}{n} u_i u_j^*\;.
	\end{align*}
	It was shown in \cite[Theorem 2.4]{FP} that $\phi$ is a complete quotient map. However, the following holds.

\begin{proposition}\cite[Proposition 3.10]{DiscGroup}
\label{proposition: tensorofquotientmapsisnotaquotient}
	The map $\phi \otimes \phi : M_2 \otimes_{\min} M_2 \to \calS(\frakz) \otimes_{\min} \calS(\frakz)$ is not a complete quotient map.
\end{proposition}

\begin{proposition}\label{DiscGroupEx}
	Let $\frakz = \{1,0,-1\} \subset \Z$. There is a $G$-equivariant complete quotient map
	\begin{align*}
		\phi: (M_2, \id, \frakz, \Z) \to (S(\frakz) , \id ,\frakz, \Z)
	\end{align*}
	which does not induce a complete quotient map on the reduced crossed product.
\end{proposition}

\begin{proof}
	Let $\phi:M_2 \to \calS(\frakz)$ be the complete quotient map as above, and suppose that the induced ucp map
	\begin{align*}
		\phi \rtimes_{\id} \frakz : M_2 \rtimes_{\id,\lambda} \frakz \to S(\frakz) \rtimes_{\id, \lambda} \frakz
	\end{align*}
	is a complete quotient map. Observe that, since the $\mathbb{Z}$-action is trivial, under the canonical isomorphisms we have $M_2 \rtimes_{\id, \lambda} \frakz = M_2 \otimes_{\min} \calS(\frakz)$ and $\calS(\frakz) \rtimes_{\id, \lambda} \frakz = \calS(\frakz) \otimes_{\min} \calS(\frakz)$. In this way, we can identify $\phi \rtimes_{\id} \frakz = \phi \otimes_{\min} \id_{S(\frakz)}$. If $\phi \rtimes_{\id} \frakz$ were a complete quotient map, then by an amplification, $\id_{M_2} \otimes_{\min} \phi$ would also be a complete quotient map. This would imply that $\phi \otimes_{\min} \phi = (\id_{M_2} \otimes_{\min} \phi) \circ(\phi \otimes_{\min} \id_{S(\frakz)})$ is a complete quotient map, contradicting Proposition~\ref{proposition: tensorofquotientmapsisnotaquotient}. Hence, $\varphi \rtimes_{\id} \mathfrak{z}$ is not a complete quotient map.

\end{proof}

\section{Full Crossed Products}

In this section we turn to the theory of full crossed products, motivated by the approach for operator algebras in \cite{katsoulisramsey}. In general, there are many choices for a relative full crossed product for operator systems.  We will focus on those regarding the smallest $C^*$-cover of an operator system (the $C^*$-envelope) and the largest $C^*$-cover (the universal $C^*$-algebra of an operator system).
\begin{definition}
  Suppose that $(\calS,G,\frakg,\alpha)$ is a dynamical system. If $(\calC,\rho,\alpha)$ is an $\alpha$-admissible $C^{\ast}$-cover of $(\calS,G,\frakg,\alpha)$, then define the \textit{full crossed product relative to} $\calC$ to be the subsystem
  \begin{align*}
    \calS \rtimes_{\alpha}^\calC \frakg := \Span\{a u_g: a \in \calS, g \in \frakg \} \subset \calC \rtimes_\alpha G\;.
  \end{align*}

  The \textit{full enveloping crossed product} of $(\calS,G,\frakg,\alpha)$ is the crossed product
  \begin{align*}
  \calS \rtimes_{\alpha,\env} \frakg := \calS \rtimes_{\alpha}^{C^*_{\env}(\calS)} \frakg.
  \end{align*}
\end{definition}

\begin{remark}
The analogue of Remark~\ref{remark: dependence on generators, reduced} holds for relative full crossed products as well. Whenever $(\calS,G,\frakg,\alpha)$ is an operator system dynamical system, $(\calC,\rho)$ is an $\alpha$-admissible $C^*$-cover and $\frakh$ is another generating set for $G$ with $\frakg \subseteq \frakh$, then there is a canonical complete order embedding $$\calS \rtimes_{\alpha}^{(\calC,\rho)} \frakg \hookrightarrow \calS \rtimes_{\alpha}^{(\calC,\rho)} \frakh.$$
\end{remark}

\begin{remark}\label{remark: connection to operator algebra, full}
  Remark~\ref{remark: connection to operator algebra, reduced} also applies to relative full crossed products. That is to say, if $\calA$ is a unital operator algebra, $G$ is a discrete group and $(\calA,G,\alpha)$ is an operator algebraic dynamical system with $\alpha$-admissible $C^*$-cover $(\calC,\rho,\alpha)$, then 
  $$(\rho(\calA)+\rho(\calA)^*) \rtimes_{\widetilde{\alpha}}^{(\calC,\widetilde{\rho})} G=(\calA \rtimes_{\alpha}^{(\calC,\rho)} G)+(\calA \rtimes_{\alpha}^{(\calC,\rho)} G)^* \subseteq \calC \rtimes_{\alpha} G.$$
\end{remark}

To define the relative full crossed product with respect to $C_u^*(\calS)$, we need the following proposition.

\begin{proposition}
  Suppose that $(\calS,G,\frakg, \alpha)$ is a dynamical system. Then there is a unique $G$-action $\overline{\alpha}$ on $C^*_u(\calS)$ which extends the action on $\calS$. Moreover, if $j:\calS \to C_u^*(\calS)$ is the canonical complete order embedding, then $(C_u^*(\calS),j,\overline{\alpha})$ is an $\alpha$-admissible $C^*$-cover for $(\calS,G,\frakg,\alpha)$.
\end{proposition}

\begin{proof}
  Suppose that $g \in G$. By the universal property of $C^*_u(\calS)$, there is a unique $*$-homomorphism $\overline{\alpha}_g$ on $C^*_u(\calS)$ for which the diagram
  \begin{center}
    \begin{tikzcd}
      C^*_u(\calS) \arrow{r}{\overline{\alpha}_g} & C^*_u(\calS) \\
      \calS \arrow{u}{j} \arrow{r}{\alpha_g} & \calS \arrow{u}{j}
    \end{tikzcd}
  \end{center}
  commutes. It is not hard to check that $\overline{\alpha}$ defines a $G$-action on $C^*_u(\calS)$.
\end{proof}

For an operator system dynamical system $(\calS,G,\frakg,\alpha)$, we will often denote the associated $G$-action on $C_u^*(\calS)$ by the same letter $\alpha$.

\begin{definition}
  Suppose that $(\calS,G,\frakg,\alpha)$ is a dynamical system. Define the \textit{full crossed product} to be the subsystem
  \begin{align*}
    \calS \rtimes_\alpha \frakg := \Span\{au_g : a \in \calS, g \in \frakg\} \subset C^*_u(\calS) \rtimes_\alpha G\;.
  \end{align*}
\end{definition}

The following fact is straightforward.

\begin{proposition}
\label{proposition: universality of full crossed product}
Let $(\calS,G,\frakg,\alpha)$ be a dynamical system, and let $(\calC,\iota)$ be any $\alpha$-admissible $C^{\ast}$-cover for $(\calS,G,\frakg,\alpha)$. Then there is a unique surjective unital $*$-homomorphism $\iota \rtimes \alpha:\calS \rtimes_{\alpha} \frakg \to \calS \rtimes_{\alpha}^{\calC} \frakg$ such that $\iota \rtimes \alpha(au_g)=\iota(a)u_g$ for all $a \in \calS$ and $g \in \frakg$.
\end{proposition}

\begin{proof}
Note that the map $\iota:\calS \to \calC$ is $G$-equivariant with respect to $\alpha$.  By the universal property of $C_u^*(\calS)$, there is a unique unital $*$-homomorphism map $\Phi:C_u^*(\calS) \to \calC$ such that $\Phi_{|\calS}=\iota$.  It is easy to see that $\Phi$ is still $G$-equivariant, so we obtain a unital $*$-homomorphism $\Phi \rtimes \alpha:C_u^*(\calS) \rtimes_{\alpha} G \to \calC \rtimes_{\alpha} G$.  Restricting to $\calS \rtimes_{\alpha} \frakg$ yields the desired map.
\end{proof}

There are two difficulties in working with full crossed products.  The first is that surjective ucp maps between operator systems are not, in general, quotient maps of operator systems.  This problem arises even in low dimensions, such as in Proposition \ref{proposition: tensorofquotientmapsisnotaquotient}.  The other key difficulty can be seen by considering any dynamical system $(\calS,G,\frakg,\alpha)$ equipped with the trivial action $\alpha=\id$.  Proposition \ref{proposition: fullenvelopingtrivialaction} below shows that $\calS \rtimes_{\id,\env} \frakg \simeq \calS \otimes_{\ess} \calS(\frakg)$, while Proposition \ref{proposition: reducedtrivialaction} shows that $\calS \rtimes_{\id,\lambda} \frakg=\calS \otimes_{\min} \calS_{\lambda}(\frakg)$.  On the other hand, the tensor product structures arising from $\calS \rtimes_{\id} \frakg$ are not as well understood.

\begin{proposition}
Let $(\calS,G,\frakg,\id)$ be a dynamical system with the trivial action.  Then $\calS \rtimes_{\id} \frakg$ is completely order isomorphic to the inclusion of the subspace $\calS \otimes \calS(\frakg) \subseteq C_u^*(\calS) \otimes_{\max} C^*(G)$.
\end{proposition}

\begin{proof}
We note that $C_u^*(\calS) \rtimes_{\id} G$ is canonically isomorphic to $C_u^*(\calS) \otimes_{\max} C^*(G)$, and that this isomorphism maps $\calS \rtimes_{\alpha} \frakg$ onto $\calS \otimes \calS(\frakg)$, which completes the proof.
\end{proof}

One could define a \emph{universal-enveloping} tensor product of operator systems $\calS$ and $\calT$ to be the operator system structure $\calS \otimes_{ue} \calT$ arising from the inclusion $\calS \otimes \calT \subseteq C_u^*(\calS) \otimes_{\max} C_{\env}^*(\calT)$.  In this way, for any trivial dynamical system $(\calS,G,\frakg,\id)$, we have $\calS \rtimes_{\id} \frakg \simeq \calS \otimes_{ue} \calS(\frakg)$.  However, the properties of this tensor product are unclear.  If $\pi_{\calS}:C_u^*(\calS) \to C_{\env}^*(\calS)$ and $\pi_{\calT}:C_u^*(\calT) \to C_{\env}^*(\calT)$ are the canonical quotient maps, then we obtain the sequence of ucp maps
\begin{center}
$\begin{tikzcd}
C_u^*(\calS) \otimes_{\max} C_u^*(\calT) \rar{\id \otimes \pi_{\calT}} & C_u^*(\calS) \otimes_{\max} C_{\env}^*(\calT) \rar{\pi_{\calS} \otimes \id} & C_{\env}^*(\calS) \otimes_{\max} C_{\env}^*(\calT) \\
\calS \otimes_c \calT \rar \uar[hook] & \calS \otimes_{ue} \calT \rar \uar[hook] & \calS \otimes_{\ess} \calT \uar[hook].
\end{tikzcd}$
\end{center}

In particular, we have $\ess \leq ue \leq c$. Similarly, one can define the \textit{enveloping-universal} tensor product of operator systems $\calS$ and $\calT$ to be the operator system structure $\calS \otimes_{eu} \calT$ arising from the inclusion $\calS \otimes \calT \subseteq C_{\env}^*(\calS) \otimes_{\max} C_u^*(\calT)$.  Clearly, the flip map $\calS \otimes \calT \to \calT \otimes \calS$ induces a complete order isomorphism $\calS \otimes_{ue} \calT \to \calT \otimes_{eu} \calS$.  On the other hand, we can at least distinguish $ue$ from $c$. To this end, we need a slight generalization of a result of Kavruk \cite[Corollary 5.8]{kavruk14}. The proof is almost identical to \cite[Proposition 3.6]{FP}; we include it for completeness.

\begin{proposition}
\label{proposition: minc with one opsys containing enough unitaries implies minmax on one with c star envelope}
Let $\calS$ be an operator system, and let $\calT \subseteq \calC$ be an operator system that contains enough unitaries in a unital $C^*$-algebra $\calC$. If $\calS \otimes_{\min} \calT=\calS \otimes_c \calT$, then $\calS \otimes_{\min} \calC=\calS \otimes_{\max} \calC$.
\end{proposition}

\begin{proof}
As every unital $C^*$-algebra is $(c,\max)$-nuclear \cite[Theorem 6.7]{KPTT}, we need only show that $\calS \otimes_{\min} \calC=\calS \otimes_c \calC$. Let $X \in M_k(\calS \otimes_{\min} \calC)$ be positive, and let $\varphi:\calS \to \bofh$ and $\psi:\calC \to \bofh$ be ucp maps with commuting ranges. We consider a minimal Stinespring representation $\psi=V^*\pi(\cdot)V$ of $\psi$ on some Hilbert space $\calH_{\pi}$.  We apply the commutant lifting theorem \cite[Theorem 1.3.1]{arveson69} to obtain a unital $*$-homomorphism $\rho:\psi(\calC)' \to \calB(\calH_{\pi})$ such that $V^*\rho(a)=aV^*$ for all $a \in \psi(\calC)'$.  The fact that $\varphi(\calS) \subseteq \psi(\calC)'$ implies that $\gamma:=\rho \circ \varphi:\calS \to \calB(\calH_{\pi})$ is ucp and its range commutes with the range of $\pi$. Since the restriction of $\gamma \cdot \pi$ to $\calS \otimes_c \calT$ is ucp and $\calS \otimes_{\min} \calT=\calS \otimes_c \calT$, it follows that $\gamma \cdot \pi$ is ucp on $\calS \otimes_{\min} \calT$.  We extend $\gamma \cdot \pi$ by Arveson's extension theorem \cite{arveson69} to a ucp map $\eta:C_{\env}^*(\calS) \otimes_{\min} \calC \to \calB(\calH_{\pi})$. Let $\{u_{\alpha}\}_{\alpha \in A}$ be a collection of unitaries in $\calT$ that generate $\calC$ as a $C^*$-algebra. Then for each $\alpha \in A$, we have $$\eta(1 \otimes u_{\alpha})=\gamma \cdot \pi(1 \otimes u_{\alpha})=\pi(u_{\alpha}),$$
which is unitary. Then each $u_{\alpha}$ is in the multiplicative domain $\calM_{\eta}$ of $\eta$, from which it follows that $1 \otimes \calC \subseteq \calM_{\eta}$. Therefore, for $s \in \calS$ and $b \in \calC$, we obtain
$$\eta(s \otimes b)=\eta(s \otimes 1)\eta(1 \otimes b)=\gamma(s)\pi(b).$$
In particular, it follows that
$$\varphi \cdot \psi(s \otimes b)=\varphi(s)\psi(b)=\varphi(s)V^*\pi(b)V=V^*\rho(\varphi(s))\pi(b)V=V^* \gamma \cdot \pi(s \otimes b)V.$$
Therefore, $\varphi \cdot \psi=V^* \eta(\cdot)_{|\calS \otimes_{\min} \calC}V$ is ucp, so that $\varphi \cdot \psi(X) \in M_k(\calB(\calH))_+$.  Hence, $X$ is positive in $M_k(\calS \otimes_c \calC)$ as well, which completes the proof.
\end{proof}

To show that $ue \neq c$, we consider the operator system
\begin{align*}
  \cW_{3,2} := \left\{\left[\begin{array}{cccccc} a & b & 0 & 0 & 0 & 0 \\
  b & a & 0 & 0 & 0 & 0 \\
  0 & 0 & a & c & 0 & 0 \\
  0 & 0 & c & a & 0 & 0 \\
  0 & 0 & 0 & 0 & a & d \\
  0 & 0 & 0 & 0 & d & a \\
\end{array} \right] : a,b,c,d \in \C \right\} \subset M_6(\C).
\end{align*}

If $\mathbb{Z}_2 * \mathbb{Z}_2 *\mathbb{Z}_2$ is the free product of three copies of $\mathbb{Z}_2$ and $h_i$ is the generator of the $i$-th copy of $\mathbb{Z}_2$, and $NC(3)$ is the operator subsystem of $C^*(*_3 \mathbb{Z}_2)$ spanned by $h_1,h_2,h_3$, then the dual operator system $NC(3)^d$ is unitally completely order isomorphic to $\cW_{3,2}$ \cite[Proposition 5.13]{DiscGroup}. Moreover, $\cW_{3,2}$ is a nuclearity detector \cite[Theorem 0.3]{kavruk15}, and since $\cW_{3,2}$ has a finite-dimensional $C^*$-cover, it follows that $C^*_{env}(\cW_{3,2})$ is nuclear. By \cite[Proposition 4.2]{guptaluthra}, $\calW_{3,2}$ is $(\min,\ess)$-nuclear.

\begin{proposition}
\label{proposition: ue neq c}
Let $G$ be a discrete group with generating set $\frakg$. If $G$ is not amenable, then $\calW_{3,2} \otimes_{ue} \calS(\frakg) \neq \calW_{3,2} \otimes_c \calS(\frakg)$.
\end{proposition}

\begin{proof}
Suppose that $\calW_{3,2} \otimes_{ue} \calS(\frakg)=\calW_{3,2} \otimes_c \calS(\frakg)$. Note that the commuting tensor product is symmetric \cite[Theorem 6.3]{KPTT}.  Thus, if $\cW_{3,2} \otimes_{ue} \calS(\frakg)$ is not completely order isomorphic to $\calS(\frakg) \otimes_{ue} \cW_{3,2}$ via the flip map, then we are done.  Suppose that this flip map is a complete order isomorphism with respect to the $ue$-tensor product.  Then $\cW_{3,2} \otimes_{ue} \calS(\frakg) \simeq \calS(\frakg) \otimes_{ue} \cW_{3,2}$.  Since $C_{\env}^*(\cW_{3,2})$ is $C^*$-nuclear and $\calS(\frakg) \otimes_{ue} \cW_{3,2} \subseteq C_u^*(G) \otimes_{\max} C_{\env}^*(\cW_{3,2})=C_u^*(G) \otimes_{\min} C_{\env}^*(\cW_{3,2})$, we see that $\calS(\frakg) \otimes_{\min} \cW_{3,2}=\calS(\frakg) \otimes_{ue} \cW_{3,2}$.  In particular, $\calS(\frakg) \otimes_{\ess} \cW_{3,2} = \calS(\frakg) \otimes_{ue} \cW_{3,2}$.  Since $\ess$ is also symmetric, applying the flip map, we have that $$\cW_{3,2} \otimes_{\ess} \calS(\frakg)=\cW_{3,2} \otimes_{ue} \calS(\frakg).$$
Since $\cW_{3,2}$ is $(\min,\ess)$-nuclear, it follows that $\cW_{3,2} \otimes_{\min} \calS(\frakg)=\cW_{3,2} \otimes_c \calS(\frakg)$.  By Proposition \ref{proposition: minc with one opsys containing enough unitaries implies minmax on one with c star envelope}, we have $\cW_{3,2} \otimes_{\min} C^*(G)=\cW_{3,2} \otimes_{\max} C^*(G)$, which implies that $C^*(G)$ is nuclear \cite[Theorem 0.3]{kavruk15}, which is a contradiction since $G$ is not amenable.  Hence, $\cW_{3,2} \otimes_{ue} \calS(\frakg) \neq \cW_{3,2} \otimes_c \calS(\frakg)$.
\end{proof}

\begin{remark}
  In the case of operator algebraic dynamical systems $(\calA, G, \alpha)$, E. Katsoulis and C. Ramsey proved \cite[Theorem 4.1]{katsoulisramsey} that
  \begin{align*}
    C^*_u(\calA \rtimes_{\alpha} G) = C^*_u(\calA) \rtimes_\alpha G\;.
  \end{align*}
  It is known that, in general, $C^*_u(\calS(\frakg))$ does not coincide with $C^*(G)$ (see \cite{DiscGroup}), so such a theorem is not expected to hold for operator system dynamical systems. In fact, $C_u^*(\calS(\frakg))$ and $C^*(G)$ fail to coincide in the case where $G=\mathbb{Z}$ and $\frakg=\{-1,0,1\}$ \cite{DiscGroup}. Moreover, since we are interested in $C^{\ast}$-envelopes, we will focus on properties of the full enveloping crossed product, rather than the full crossed product.
\end{remark}

Like the reduced crossed product, the full enveloping crossed product preserves hyperrigidity. The proof is exactly the same as in Theorem \ref{theorem: C*-envelope of reduced}, so we omit it. For a proof in the operator algebraic case, see \cite[Theorem 2.7]{katsoulis17}.

\begin{theorem}
\label{theorem: hyperrigidity preserved by full enveloping}
  Suppose that $(\calS,G,\frakg, \alpha)$ is a dynamical system. Suppose that $\calS$ is hyperrigid. Then the full enveloping crossed product is hyperrigid. In particular, $C^*_{\env}(\calS \rtimes_{\alpha,\env} \frakg) \simeq C^*_{\env}(\calS) \rtimes_{\alpha} G$.
\end{theorem}

We now give the tensor product description of the full enveloping crossed product with respect to a trivial action.

\begin{proposition}
\label{proposition: fullenvelopingtrivialaction}
  Suppose that $(\calS, G, \frakg, id)$ is a trivial dynamical system. We have the isomorphism
  \begin{align*}
    \calS \rtimes_{id,\env} \frakg &\simeq \calS \otimes_{\ess} \calS(\frakg)\;.
  \end{align*}
\end{proposition}

\begin{proof}
  We know that
  \begin{align*}
    C^*_{\env}(\calS) \rtimes_{id} G &\simeq C^*_{\env}(\calS) \otimes_{\max} C^*(G)\;.
  \end{align*}
  This induces an isomorphism
  \begin{align*}
    \calS \rtimes_{id,\env} \frakg &\simeq \Span\{a \otimes u_g \in C^*_{\env}(\calS) \otimes_{\max} C^*(G): a \in \calS, g \in \frakg \}\;.
  \end{align*}
  By definition, the span on the right hand side is $\calS \otimes_{\ess} \calS(\frakg)$.
\end{proof}

For amenable groups, there is no difference between the reduced and the full enveloping crossed products.

\begin{proposition}
  Suppose that $(\calS,G,\frakg,\alpha)$ is a dynamical system with $G$ amenable. Then $\calS \rtimes_{\alpha,\lambda} \frakg = \calS \rtimes_{\alpha,\env} \frakg$.
\end{proposition}

\begin{proof}
  Since $G$ is amenable, we have the isomorphism $C^*_{\env}(\calS) \rtimes_{\alpha,\lambda} G = C^*_{\env}(\calS) \rtimes_{\alpha} G$ sending generators to generators. By restricting this isomorphism to $\calS \rtimes_{\alpha,\lambda} \frakg$, we get the identity $\calS \rtimes_{\alpha,\lambda} \frakg = \calS \rtimes_{\alpha,\env} \frakg$.
\end{proof}

\section{Two Problems of Katsoulis and Ramsey}

Before proving Theorems \ref{Theorem: intro counterexample problem 2 KR} and \ref{Theorem: intro counterexample problem 1 KR}, we need some facts about the non self-adjoint operator algebra $\calU(X)$ of an operator space $X$. If $X$ is an operator space in $B(\calH)$, then define the associated non self-adjoint operator algebra $\calU(X)$ (see \cite[Sections 2.2.10-2.2.11]{blecherlemerdy} for more on this algebra) as the subalgebra of $B(H^{2}) \simeq M_2(B(H))$ given by
\begin{align*}
  \calU(X) := \left\{ \left[\begin{array}{cc} \lambda & x \\ 0 & \lambda \end{array} \right]: \lambda \in \C, \, x \in X \right\} \;.
\end{align*}
Note that the algebra $\calU(X)$ does not depend on the representation chosen for $X$.

Recall that an operator system $\calS$ \emph{detects nuclearity} (or is a \emph{nuclearity detector}) if, whenever $\calD$ is a unital $C^{\ast}$-algebra and
\begin{align*}
  \calS \otimes_{\min} \calD = \calS \otimes_{c} \calD,
\end{align*}
then $\calD$ is a nuclear $C^{\ast}$-algebra \cite{kavruk15}. Our first goal is to show that if $\calS$ is a nuclearity detector, then $\calU(\calS)$ is a nuclearity detector for $C^*$-algebras; that is, if $\calU(\calS) \otimes_{\min} \calD = \calU(\calS) \otimes_{\max} \calD$ then $\calD$ is a nuclear $C^{\ast}$-algebra.

We first wish to interpret the relevant tensor products of an operator system $\calS$ with a unital $C^*$-algebra as an operator space. We know that the norm arising from the operator system $\calS \otimes_{\min} \calD$ agrees with the minimal operator space tensor norm \cite[Corollary 4.9]{KPTT}. For the norm arising from the commuting tensor product, we have the following result.

\begin{lemma}
\label{commutingnorm}
  Let $\calS$ be an operator system; let $\calD$ be a unital $C^{\ast}$-algebra; and let $d \in \mathbb{N}$. For any element $Z \in M_d(\calS \otimes_c \calD) \simeq \calS \otimes_c M_d(\calD)$, we have
  \begin{align*}
    \|Z\|_{M_d(\mathcal{S} \otimes_c \mathcal{D})} = \sup_{\theta, \pi} \|\theta \cdot \pi(Z)\|
  \end{align*}
  where the supremum is taken over all pairs of commuting maps
  \begin{align*}
    \theta &: \calS \to B(H) \text{ and}\\
    \pi &: M_d(\calD) \to B(H),
  \end{align*}
  where $\theta$ is ucp and $\pi$ is a unital $*$-homomorphism.
\end{lemma}

\begin{proof}
  Since $M_d(\mathcal{D})$ is itself a unital $C^{\ast}$-algebra, by replacing $M_d(\mathcal{D})$ with $\mathcal{D}$ if necessary, we may assume that $d=1$. We know that the order structure of $\calS \otimes_c \calD$ is inherited from the order structure of $C^*_u(\calS) \otimes_{\max} \calD$. From this fact, it follows that the norm structure of $\calS \otimes_c \calD$ is also inherited from the norm structure of $C^*_u(\calS) \otimes_{\max} \calD$. Given an element $z \in \calS \otimes_c \calD$, we have
  \begin{align*}
    \|z\|_c = \sup_{\theta', \pi'} \|\theta' \cdot \pi'(z)\|,
  \end{align*}
  where the supremum is taken over all unital $*$-homomorphisms $\theta': C^*_u(\calS) \to B(H)$ and $\pi': \calD \to B(H)$, with commuting ranges. Since every unital $*$-homomorphism $\theta: C^*_u(\calS) \to B(H)$ restricts to a ucp map on $\calS$, we obtain the inequality
  \begin{align*}
    \|z\|_c \leq \sup_{\theta, \pi} \|\theta \cdot \pi(z)\|,
  \end{align*}
  where the supremum is taken over all ucp maps $\theta:\calS \to B(\calH)$ and unital $*$-homomorphisms $\pi:\calD \to B(\calH)$ with commuting ranges. Conversely, if $\theta: \calS \to B(H)$ is ucp and $\pi: \calD \to B(H)$ is a unital $*$-homomorphism whose range commutes with the range of $\theta$, then by the universal property of $C_u^*(\calS)$, there is a unital $*$-homomorphism $\theta' : C^*_u(\calS) \to B(H)$ such that $\theta'_{|\calS}=\theta$. Since the range of $\theta$ commutes with the range of $\pi$ and and $C_u^*(\calS)$ is generated as a $C^*$-algebra by $\calS$, we have that $\theta'$ and $\pi$ have commuting ranges. Moreover, $\|\theta' \cdot \pi(z)\| = \|\theta \cdot \pi(z)\|$. Therefore, the reverse inequality holds, and the result follows.
\end{proof}

The next lemma is the operator system analogue of a result of Blecher and Duncan \cite[Lemma 6.3]{blecherduncan}.

\begin{lemma}
\label{lemma: operator system nuclearity detectors to operator algebra nuclearity detectors}
  Let $\calS$ be an operator system and let $\calD$ be a unital $C^*$-algebra. If $\calU(\calS)\otimes_{\min}\calD = \calU(\calS) \otimes_{\max} \calD$, then $\calS \otimes_{\min} \calD = \calS \otimes_c \calD$. In particular, if $\calS$ is a nuclearity detector, then $\calU(\calS)$ is a nuclearity detector.
\end{lemma}

\begin{proof}
  Suppose that $\calU(\calS) \otimes_{\min} \calD = \calU(\calS) \otimes_{\max} \calD$, and let $z \in \calS \otimes \calD$. Since we always have $\|z\|_c \geq \|z\|_{\min}$, it suffices to show that $\|z\|_c \leq \|z\|_{\min}$. We may write
  \begin{align*}
    z = \sum_{i=1}^k z_i \otimes d_i,
  \end{align*}
  where $z_i \in \calS$ and $d_i \in \calD$. Let $\theta: \calS \to B(\calH)$ be ucp and $\pi: \calD \to B(\calH)$ be a unital $*$-homomorphism whose range commutes with the range of $\theta$. We define
  \begin{align*}
    \theta': \calU(\calS) \to B(\calH^2) : \left[\begin{array}{cc} \lambda & x \\ 0 & \lambda \end{array} \right] \mapsto \left[\begin{array}{cc} \lambda & \theta(x) \\ 0 & \lambda \end{array} \right].
  \end{align*}
  It follows from \cite[Lemma 8.1]{paulsenbook} that $\theta'$ is a unital completely contractive map. A calculation also shows that $\theta'$ is a homomorphism. The amplification $\pi \oplus \pi: \calD \to B(H^2)$ is a unital $*$-homomorphism such that $\theta'$ and $\pi \oplus \pi$ have commuting ranges. In particular, by definition of the maximal operator algebra tensor norm,
  \begin{align*}
    \|\theta \cdot \pi(z)\| &= \left\| \theta' \cdot (\pi \oplus \pi)\left(\sum_i \left[\begin{array}{cc} 0 & z_i \\ 0 & 0 \end{array} \right] \otimes d_i \right) \right\| \\
    &\leq \left\|\sum_i \left[\begin{array}{cc} 0 & z_i \\ 0 & 0 \end{array} \right] \otimes d_i \right\|_{\max} \\
    &= \left\|\sum_i \left[\begin{array}{cc} 0 & z_i \\ 0 & 0 \end{array} \right] \otimes d_i \right\|_{\min} \\
    &= \|z\|_{\min}
  \end{align*}
  Since this is for arbitrary pairs $(\theta,\pi)$, it follows that $\|z\|_c \leq \|z\|_{\min}$. A similar argument shows that $\|Z\|_c=\|Z\|_{\min}$ for every $Z \in M_d(\calS \otimes \calD)$ and $d \geq 1$. Thus, $\calS \otimes_{\min} \calD=\calS \otimes_c \calD$.

  If $\calS$ is a nuclearity detector and $\calD$ is a unital $C^{\ast}$-algebra such that $\calU(\calS) \otimes_{\min} \calD = \calU(\calS) \otimes_{\max} \calD$, then the above proof shows that $\calS \otimes_{\min} \calD = \calS \otimes_c \calD$. Since $\calS$ is a nuclearity detector, $\calD$ must be a nuclear $C^*$-algebra.
\end{proof}

Recall that
\begin{align*}
  \cW_{3,2} := \left\{\left[\begin{array}{cccccc} a & b & 0 & 0 & 0 & 0 \\
  b & a & 0 & 0 & 0 & 0 \\
  0 & 0 & a & c & 0 & 0 \\
  0 & 0 & c & a & 0 & 0 \\
  0 & 0 & 0 & 0 & a & d \\
  0 & 0 & 0 & 0 & d & a \\
\end{array} \right] : a,b,c,d \in \C \right\} \subset M_6(\C).
\end{align*}
Note that $C_{\env}^*(\calW_{3,2})$ is nuclear. In particular, $\calU(\cW_{3,2})$ has a nuclear $C^*$-envelope as well. On the other hand, since $\calW_{3,2}$ is a nuclearity detector \cite[Theorem 0.3]{kavruk15}, Lemma \ref{lemma: operator system nuclearity detectors to operator algebra nuclearity detectors} shows that $\calU(\calW_{3,2})$ is a nuclearity detector. Unlike the situation for operator systems, the operator algebra $\calU(\cW_{3,2})$ also detects nuclearity for non-unital $C^*$-algebras. To show this fact, we first need the following fact.

\begin{proposition}
\label{proposition: unital opalg tensor cstaralg is contained in unitization}
Let $\tau$ be the minimal or maximal operator algebra tensor product. Let $\calA$ be a unital operator algebra and let $\calD$ be a $C^*$-algebra. If $\calD^+$ is the minimal unitization of $\calD$ and $\pi:\calD^+ \to \calD^+/\calD \simeq \bC$ is the canonical quotient map, then the sequence
\begin{center}
$\begin{tikzcd}
0 \rar & \calA \otimes_{\tau} \calD \rar & \calA \otimes_{\tau} \calD^+ \rar{\id_{\calA} \otimes \pi} & \calA \rar & 0
\end{tikzcd}$
\end{center}
is exact. In other words, $\calA \otimes_{\tau} \calD \subseteq \calA \otimes_{\tau} \calD^+$ completely isometrically and $\id_{\calA} \otimes \pi$ is a complete quotient map of $\calA \otimes_{\tau} \calD^+$ onto $\calA$.
\end{proposition}

\begin{proof}
The inclusion $\calA \otimes_{\min} \calD \subseteq \calA \otimes_{\min} \calD^+$ holds by definition of the minimal tensor product. The fact that $\calA \otimes_{\max} \calD \subseteq \calA \otimes_{\max} \calD^+$ completely isometrically follows by \cite[Proposition 2.5]{blecherduncan}.

The tensor product map $\id_{\calA} \otimes \pi:\calA \otimes_{\tau} \calD^+ \to \calA \otimes_{\tau} \bC=\calA$ is completely contractive. On the other hand, the map $\phi:\calA \to \calA \otimes_{\tau} \calD^+$ given by $\phi(a)=a \otimes 1_{\calD^+}$ is a completely contractive unital homomorphism.  Hence, $\phi$ is a completely contractive splitting of $\id_{\calA} \otimes \pi$. It follows that $\id_{\calA} \otimes \pi$ is a complete quotient map.
\end{proof}

\begin{lemma}
\label{lemma: tensor from nonunital cstar to unitization}
Let $\calA$ be a unital operator algebra; let $\calD$ be a non-unital $C^*$-algebra, and let $\calD^+$ be its minimal unitization. Then $\calA \otimes_{\min} \calD=\calA \otimes_{\max} \calD$ if and only if $\calA \otimes_{\min} \calD^+=\calA \otimes_{\max} \calD^+$.
\end{lemma}

\begin{proof}
If $\calA \otimes_{\min} \calD^+=\calA \otimes_{\max} \calD^+$, then applying Proposition \ref{proposition: unital opalg tensor cstaralg is contained in unitization} shows that $\calA \otimes_{\tau} \calD \subseteq \calA \otimes_{\tau} \calD^+$ completely isometrically for $\tau \in \{\min,\max\}$. It immediately follows that $\calA \otimes_{\min} \calD=\calA \otimes_{\max} \calD$.

Conversely, suppose that $\calA \otimes_{\min} \calD=\calA \otimes_{\max} \calD$. Using the canonical maps, the following diagram commutes:
\begin{center}
$\begin{tikzcd}
0 \rar & \calA \otimes_{\max} \calD \dar \rar & \calA \otimes_{\max} \calD^+ \dar \rar & \calA \dar \rar & 0 \\
0 \rar & \calA \otimes_{\min} \calD \rar & \calA \otimes_{\min} \calD^+ \rar & \calA \rar & 0
\end{tikzcd}$
\end{center}
By Proposition \ref{proposition: unital opalg tensor cstaralg is contained in unitization}, the rows are exact.  By the Five-Lemma for operator algebras \cite[Lemma 3.2]{blecherroyce}, since the outer two vertical arrows are complete isometries, the middle arrow is a complete isometry. Thus, $\calA \otimes_{\min} \calD^+=\calA \otimes_{\max} \calD^+$.
\end{proof}

\begin{lemma}
  Suppose that $\calA$ is a unital operator algebra for which for every unital $C^{\ast}$-algebra $\calC$, $\calA \otimes_{\min} \calC = \calA \otimes_{\max} \calC$ if and only if $\calC$ is a nuclear $C^{\ast}$-algebra. Then for every non-unital $C^{\ast}$-algebra $\calD$, $\calA \otimes_{\min} \calD = \calA \otimes_{\max} \calD$ if and only if $\calD$ is a nuclear $C^{\ast}$-algebra.
\end{lemma}

\begin{proof}
  Let $\calD$ be a non-unital $C^{\ast}$-algebra. Let $\calD^+$ be its minimal unitization. We know that $\calD$ is nuclear if and only if $\calD^+$ is nuclear. Thus, $\calD$ is nuclear if and only if $\calA \otimes_{\min} \calD^+ = \calA \otimes_{\max} \calD^+$. By Lemma \ref{lemma: tensor from nonunital cstar to unitization}, the latter condition is equivalent to having $\calA \otimes_{\min} \calD = \calA \otimes_{\max} \calD$, as desired.
\end{proof}

We are now in a position to prove Theorem \ref{Theorem: intro counterexample problem 2 KR}, which gives a counterexample to the second problem of Katsoulis and Ramsey.

\begin{theorem}
\label{KRexamplelocallycompactgroup}
  Suppose that $G$ is a locally compact group such that $C^*_\lambda(G)$ admits a tracial state. Let $\calA := \calU(\cW_{3,2})$ be the operator subalgebra of $M_{12}(\C)$ endowed with the trivial $G$-action $\id: G \acts \calA$. The following are equivalent:
  \begin{enumerate}
    \item $\calA \rtimes_{C^*_{env}(\calA),\id} G = \calA \rtimes_{C^*_u(\calA),\id} G$.
    \item The group $G$ is amenable.
  \end{enumerate}
\end{theorem}

\begin{proof}
  The implication $(2) \to (1)$ holds by \cite[Theorem 3.14]{katsoulisramsey}. Conversely, suppose we have the identity $\calA \rtimes_{C^*_{env}(\calA),\id} G = \calA \rtimes_{C^*_u(\calA),\id} G$. Since $C_{\env}^*(\calA)$ is nuclear,
  \begin{align*}
    C^*_{env}(\calA) \otimes_{\min} C^*(G) = C^*_{env}(\calA) \otimes_{\max} C^*(G)=C_{\env}^*(\calA) \rtimes_{\id} G\;.
  \end{align*}
  Thus, $\calA \rtimes_{C_{\env}^*(\calA),\id} G$ is completely isometrically isomorphic to the completion of $\calA \otimes C^*(G)$ in $C_{\env}^*(\calA) \otimes_{\min} C^*(G)$. We conclude that $$\calA \rtimes_{C_{\env}^*(\calA,\id)} G=\calA \otimes_{\min} C^*(G).$$
  On the other hand, by Example \ref{example: opalg trivial actions}, we have $$\calA \rtimes_{C_u^*(\calA),\id} G=\calA \otimes_{\max} C^*(G).$$
  Then we have the identity
  \begin{align*}
    \calA \otimes_{\min} C^*(G) = \calA \rtimes_{C^*_{env}(\calA),\id} G  = \calA \rtimes_{C^*_u(\calA),\id} G =  \calA \otimes_{max} C^*(G)\;.
  \end{align*}
  Since $\calA$ is a nuclearity detector, $C^*(G)$ is nuclear. In particular, since nuclearity of $C^*$-algebras passes to quotients \cite{choieffros}, it follows that $C^*_\lambda(G)$, which  is a quotient of $C^*(G)$, must be nuclear. By \cite[Corollary 3.3]{FSW}, $G$ is amenable if and only if $C^*_\lambda(G)$ is nuclear and $C^*_\lambda(G)$ admits a tracial state. Therefore, $G$ is amenable.
\end{proof}

Theorem \ref{KRexamplelocallycompactgroup} allows for a proof of Theorem \ref{Theorem: intro counterexample problem 1 KR}, which gives a counterexample to the first problem of Katsoulis and Ramsey.

\begin{theorem}
\label{KRproblem1}
  Let $\calA = \calU(\cW_{3,2})$. If $G$ is a discrete group, then the following are equivalent:
  \begin{enumerate}
    \item We have the identity $C^*_{env}(\calA) \rtimes_{id} G = C^*_{env}(\calA \rtimes_{\id} G)$.
    \item The group $G$ is amenable.
  \end{enumerate}
\end{theorem}

\begin{proof}
  The proof that (2) implies (1) was done by Katsoulis \cite[Theorem 2.5]{katsoulis17}. For the converse, if $C^*_{env}(\calA)~\rtimes_{\id}~G = C^*_{env}(\calA \rtimes_{\id} G)$, then $\calA \rtimes_{C_u^*(\calA),\id} G$ embeds faithfully and canonically into $C^*_{env}(\calA) \rtimes_{\id} G$. Thus, $\calA \rtimes_{C^*_{env}(\calA), \id} G = \calA \rtimes_{C^*_{u}(\calA),\id} G$. By Theorem \ref{KRexamplelocallycompactgroup}, this implies that $G$ is amenable.
\end{proof}

\section{Hyperrigidity and $\calU(\calW_{3,2})$}

In this final section, we show that $(\min,\ess)$-nuclearity is preserved under the full enveloping crossed product whenever the operator system is hyperrigid.  We also show that $\calU(\calW_{3,2})$ is hyperrigid, which shows that the equation $C_{\env}^*(\calA \rtimes_{\alpha} G)=C_{\env}^*(\calA) \rtimes_{\alpha} G$ can fail even in the case of a hyperrigid operator algebra.

Let $\calS \subseteq \calC$ and $\calT \subseteq \calD$ are hyperrigid operator subsystems of unital $C^*$-algebras $\calC$ and $\calD$, respectively. By injectivity of the minimal tensor product, $\calS \otimes_{\min} \calT \subseteq \calC \otimes_{\min} \calD$. Moreover, by Theorem \ref{theorem: hyperrigid is envelope}, $C_{\env}^*(\calS)=\calC$ and $C_{\env}^*(\calT)=\calD$. In particular, by definition of the essential tensor product, we have that $\calS \otimes_{\ess} \calT \subseteq \calC \otimes_{\max} \calD$. In fact, more is true.

\begin{lemma}
  Suppose that $\calS \subset \calC$ and $\calT \subset \calD$ are hyperrigid operator subsystems of unital $C^{\ast}$-algebras $\calC$ and $\calD$ respectively. Then $\calS \otimes_{\min} \calT \subset \calC \otimes_{\min} \calD$ and $\calS \otimes_{\ess} \calT \subset \calC \otimes_{\max} \calD$ are hyperrigid.
\end{lemma}

\begin{proof}
  We prove the result for the minimal tensor product; the proof for the essential tensor product is the same. Let $\iota_\calC : \calC \to \calC \otimes_{\min} \calD : a \mapsto a \otimes 1$ and $\iota_\calD : \calD \to \calC \otimes_{\min} \calD : b \mapsto 1 \otimes b$ be the canonical *-homomorphisms. Suppose that $\pi: \calC \otimes_{\min} \calD \to B(\calH)$ is a $*$-representation and suppose that $\rho: \calC \otimes_{\min} \calD \to B(\calH)$ is a ucp map extending the map $\pi|_{\calS \otimes_{\min} \calT}$. Since $\rho \circ \iota_\calC$ agrees with $\pi \circ \iota_\calC$ on $\calS$, by hyperrigidity of $\calS$, $\rho \circ \iota_\calC = \pi \circ \iota_\calC$. Similarly, $\rho \circ \iota_\calD = \pi \circ \iota_\calD$. For any $a \in \calC$ and $b \in \calD$, since $\iota_\calD(b)$ is in the multiplicative domain of $\rho$,
  \begin{align*}
    \rho(a \otimes b) = \rho(\iota_\calC(a) \iota_\calD(b)) = \rho(\iota_\calC(a)) \rho(\iota_\calD(b)) = \pi(a \otimes b)\;.
  \end{align*}
  Extending by linearity and continuity shows that $\rho=\pi$.
\end{proof}

The following proposition is a generalization of a result of Gupta and Luthra \cite[Theorem 4.3]{guptaluthra}.
\begin{proposition}\label{prop: nuclear passes to envelope}
  Suppose that $\calS \subset \calC$ is a hyperrigid operator system. Then $\calS$ is $(\min,\ess)$-nuclear if and only if $\calC$ is nuclear.
\end{proposition}

\begin{proof}
  If $\calC$ is nuclear, then by \cite[Proposition 4.2]{guptaluthra}, $\calS$ is $(\min,\ess)$-nuclear. Conversely, suppose that $\calS$ is $(\min,\ess)$-nuclear, and let $\calD$ be any unital $C^*$-algebra. First, let us show that $C^*_{\env}(\calS \otimes_{\min} \calD) = \calC \otimes_{\min} \calD$ and $C^*_{\env}(\calS \otimes_{\ess} \calD) = \calC \otimes_{\max} \calD$. By Theorem \ref{theorem: hyperrigid is envelope}, it suffices to show that $\calS \otimes_{\min} \calD$ and $\calS \otimes_{\max} \calD$ are hyperrigid in their $C^{\ast}$-covers. For the minimal case, suppose that
  \begin{align*}
  	\pi: \calC \otimes_{\stmin} \calD \to B(H)
  \end{align*}
  is a unital $*$-homomorphism and suppose that $\phi$ is the restriction of $\pi$ to $\calS \otimes_{\min} \calD$. If $\rho$ is any ucp extension of $\phi$ to $\calC \otimes_{\min} \calD$, then the ucp map $\rho|_\calC$ is a ucp extension of $\phi|_\calS$. By hyperrigidity of $\calS$, $\rho|_\calC= \pi$. As well, $\rho|_\calD = \phi|_\calD =\pi|_\calD$. Thus, $\calC \otimes 1$ and $1 \otimes \calD$ are in the multiplicative domain of $\rho$. Therefore, for all $c \in \calC$ and $d \in \calD$,
  $$\rho(c \otimes d)=\rho(c \otimes 1)\rho(1 \otimes d)=\pi(c \otimes 1)\pi(1 \otimes d)=\pi(c \otimes d).$$
  Extending by linearity and continuity, we have $\rho = \pi$ on $\calC \otimes_{\min} \calD$. The proof for the maximal tensor product is the same, replacing all tensors with the appropriate type. Hence, $\calS \otimes_{\min} \calD$ is hyperrigid in $\calC \otimes_{\min} \calD$, and $\calS \otimes_{\ess} \calD$ is hyperrigid in $\calC \otimes_{\max} \calD$.

Since $\calS \otimes_{\min} \calD = \calS \otimes_{\ess} \calD$, and both operator systems are hyperrigid, we get
\begin{align*}
	\calC \otimes_{\min} \calD = C^*_{\env}(\calS \otimes_{\min} \calD)= C^*_{\env}(\calS \otimes_{\ess} \calD)=\calC \otimes_{\max} \calD\;.
	\end{align*}
	As $\calD$ was an arbitrary unital $C^*$-algebra, it follows that $\calC$ is nuclear.
\end{proof}

\begin{corollary}
  Suppose that $(\calS, G,\frakg,\alpha)$ is a dynamical system with $\calS$ hyperrigid in $C_{\env}^*(\calS)$ and $G$ amenable. Then $\calS$ is $(\min,\ess)$-nuclear if and only if $\calS \rtimes_{\alpha,\env} \frakg$ is $(\min,\ess)$-nuclear.
\end{corollary}

\begin{proof}
  By Proposition~\ref{prop: nuclear passes to envelope}, $\calS$ is (min,ess)-nuclear if and only if $C^*_{\env}(\calS)$ is nuclear. Since $G$ is discrete and amenable, $C^*_{\env}(\calS)$ is nuclear if and only if $C^*_{\env}(\calS) \rtimes_\alpha G$ is nuclear \cite[Theorem 4.2.6]{BrownAndOzawa}. Using Theorem \ref{theorem: hyperrigidity preserved by full enveloping}, $\calS \rtimes_{\alpha,\env} \frakg$ is hyperrigid in $C_{\env}^*(\calS) \rtimes_{\alpha} G$. Applying Proposition~\ref{prop: nuclear passes to envelope} again, $C^*_{\env}(\calS) \rtimes_\alpha G$ is nuclear if and only if $\calS \rtimes_{\alpha,\env} \frakg$ is $(\min,\ess)$-nuclear.
\end{proof}

We now will work towards showing that $\calU(\calW_{3,2})$ is hyperrigid in its $C^*$-envelope. We begin with the following helpful fact about $C_{\env}^*(\calU(\calS))$.

\begin{proposition}
\label{proposition: cover for S gives cover for U(S)}
Let $\calS$ be an operator system with $C^*$-cover $\calD$.  Then $M_2(\calD)$ is a $C^*$-cover for $\calU(\calS)$. Moreover, $C_{\env}^*(\calU(\calS))=M_2(C_{\env}^*(\calS))$.
\end{proposition}

\begin{proof}
Let $\varphi:\calS \to \calD$ be a complete order isomorphism, where $\calD$ is a unital $C^*$-algebra such that $\calD=C^*(\varphi(\calS))$. Then there is an associated unital, completely isometric homomorphism $\psi:\calU(\calS) \to M_2(\calD)$ such that $$\psi \left( \begin{bmatrix} \lambda & x \\ 0 & \lambda \end{bmatrix} \right)=\begin{bmatrix} \lambda & \varphi(x) \\ 0 & \lambda \end{bmatrix}$$
for all $\lambda \in \bC$ and for all $x \in \calS$. The matrix units $E_{12},E_{21}$ are in $C^*(\psi(\calU(\calS)))$ since $\varphi$ is unital. Hence, every element of the form $E_{ii} \otimes \varphi(x)$, for $x \in S$, is in $C^*(\psi(\calU(\calS)))$. Thus, $E_{ii} \otimes a \in C^*(\psi(\calU(\calS)))$ for all $a \in \calD$.  Using the matrix units $E_{12}$ and $E_{21}$, we see that $C^*(\psi(\calU(\calS)))=M_2(\calD)$. Thus, $M_2(\calD)$ is a $C^*$-cover for $\calU(\calS)$ whenever $\calD$ is a $C^*$-cover for $\calS$.

Lastly, we must show that $C_{\env}^*(\calU(\calS))=M_2 \left( C_{\env}^*(\calS) \right)$. To this end, we let $\rho:M_2 \left( C_{\env}^*(\calS) \right) \to C_{\env}^*(\calU(\calS))$ be a surjective, unital $*$-homomorphism that preserves the copy of $\calU(\calS)$. Let $\gamma$ be the restriction of $\rho$ to the subalgebra $M_2 \otimes 1_{C_{\env}^*(\calS)}$. Since $\gamma$ is a unital $*$-homomorphism and $M_2$ is simple, $\gamma$ must be injective. In particular, if $\calD=\rho(I_2 \otimes C_{\env}^*(\calS))$, then $M_2$ and $\calD$ are commuting unital $C^*$-subalgebras of $C_{\env}^*(\calU(\calS))$ that generate the whole algebra. Let $\eta$ be the restriction of $\rho$ to $I_2 \otimes \left( C_{\env}^*(\calS) \right)$. Since $\rho$ is a completely isometric homomorphism on $\calU(\calS)$, a standard canonical shuffle argument \cite[p.~97]{paulsenbook} shows that $\rho$ must be a complete isometry when restricted to $E_{12} \otimes \calS$. As $\rho$ is multiplicative, we see that $\eta$ is a complete isometry on $I_2 \otimes \calS$. Thus, $\eta_{|I_2 \otimes \calS}$ maps $\calS$ completely order isomorphically into $\calD$.  Since $I_2 \otimes \calS$ generates $I_2 \otimes C_{\env}^*(\calS)$ as a $C^*$-algebra, the image of $I_2 \otimes \calS$ under $\eta$ generates $\calD$. By the definition of the $C^*$-envelope, there is a unique, surjective unital $*$-homomorphism $\delta:\calD \to C_{\env}^*(\calS)$ such that $\delta(\eta(I_2 \otimes s))=s$ for all $s \in \calS$. The map $\eta:C_{\env}^*(\calS)=I_2 \otimes C_{\env}^*(\calS) \to \calD$ also fixes the copy of $\calS$.  Hence, $\delta \circ \eta$ and $\eta \circ \delta$ are the identity when restricted to $\calS$.  As $C_{\env}^*(\calS)$ and $\calD$ are generated by their respective copies of $\calS$, it follows that $\delta$ and $\eta$ are inverses of each other.  Therefore, $\calD \simeq C_{\env}^*(\calS)$.  We conclude that $C_{\env}^*(\calU(\calS))=M_2(C_{\env}^*(\calS))$.
\end{proof}

The following shows that hyperrigidity of $\calU(S)$ passes to $\calS$.

\begin{proposition}
\label{proposition: hyperrigid passes from Paulsen to system}
  Let $\calD$ be a $C^{\ast}$-cover of $S$. If $\calU(S)$ is hyperrigid in $M_2(\calD)$, then $S$ is hyperrigid in $\calD$.
\end{proposition}

\begin{proof}
  Assume without loss of generality that $\calS \subset \calD$. Suppose that $\pi: \calD \to B(\calH)$ is a unital $*$-homomorphism with restriction $\rho=\pi_{|\calS}: \calS \to B(\calH)$ to $\calS$. Let $\phi: \calD \to B(\calH)$ be any ucp extension of $\rho$. Then $\pi^{(2)}:M_2(\calD) \to B(\calH)$ is a unital $*$-homomorphism.  Moreover, if $\eta=(\pi^{(2)})_{|\calU(\calS)}$, then $\phi^{(2)}:M_2(\calD) \to B(\calH)$ is a ucp extension of $\eta$, since $\varphi_{|\calS}=\rho$. By hyperrigidity of $\calU(\calS)$, we have $\phi^{(2)}=\pi^{(2)}$. In particular, $\varphi=\pi$, as desired.
\end{proof}

\begin{corollary}
\label{corollary: hyperrigidity of U(S) in M2(D) tells cstar envelope of S and U(S)}
Let $\calS$ be an operator system with $C^*$-cover $\calD$. If $\calU(\calS)$ is hyperrigid in $M_2(\calD)$, then $C_{\env}^*(\calS)=\calD$ and $C_{\env}^*(\calU(\calS))=M_2(\calD)$.
\end{corollary}

\begin{proof}
By Proposition \ref{proposition: hyperrigid passes from Paulsen to system}, since $\calU(\calS)$ is hyperrigid in $M_2(\calD)$, $\calS$ is hyperrigid in $\calD$.  By Theorem \ref{theorem: hyperrigid is envelope}, $C_{\env}^*(\calS)=\calD$. The fact that $C_{\env}^*(\calU(\calS))=M_2(\calD)$ follows by Proposition \ref{proposition: cover for S gives cover for U(S)}.
\end{proof}

Let $\calS$ be an operator subsystem of a $C^{\ast}$-algebra $\calD$. Suppose that any irreducible representation $\pi: \calD \to B(\calH)$ of $\calD$ is maximal on $S$. In general, it is unknown whether this implies $S$ is hyperrigid in $\calD$. In our case however, this obstruction is not a concern.

\begin{proposition}
\label{proposition: boundary hyperrigid for matrices}
  Let $\calS$ be an operator system with finite-dimensional $C^*$-cover $\calF$. If every irreducible representation of $\calF$ is maximal on $\calS$, then $S$ is hyperrigid in $\calF$.
\end{proposition}

\begin{proof}
Since $\calF$ is finite-dimensional, it has finite spectrum; i.e., there are only finitely many irreducible representations of $\calF$ up to unitary equivalence. The desired result then follows by a theorem of Arveson \cite[Theorem 5.1]{arveson11}.
\end{proof}

\begin{lemma}
\label{hyperrigidopalgiffaplusastar}
  If $\calA$ is a unital operator algebra with $C^{\ast}$-cover $\calD$, and $\calS = \calA + \calA^*$, then $\calA$ is hyperrigid in $\calD$ if and only if $\calS$ is hyperrigid in $\calD$.
\end{lemma}

\begin{proof}
  If $\calA$ is hyperrigid in $\calD$ and $\pi:\calD \to B(\calH)$ is a unital $*$-homomorphism, then any ucp extension $\varphi:\calD \to B(H)$ of $\pi_{|\calS}:\calS \to B(\calH)$ is necessarily a unital, completely contractive extension of $\pi_{|\calA}$, so that $\varphi=\pi$.  Hence, $\calS$ is hyperrigid in $\calD$. Similarly, if $\calS$ is hyperrigid in $\calD$ and $\pi:\calD \to B(\calH)$ is a unital $*$-homomorphism, then any unital completely contractive extension $\psi:\calD \to B(\calH)$ of $\pi_{|\calA}:\calA \to B(\calH)$ is ucp and satisfies $\psi(a+a^*)=\pi(a)+\pi(a)^*$ for all $a \in \calA$.  Thus, $\psi$ is a ucp extension of $\pi_{|\calS}$, so that $\psi=\pi$, which establishes the converse direction.
\end{proof}

For notational convenience, we let $$\mathcal{B}=\left\{ \left[\begin{array}{cc} a & b \\ b & a \end{array}\right]: a,b \in \C\right\}.$$
Note that $\calB \simeq \C^2$ as $C^*$-algebras; however, the given presentation of $\calB$ will be most useful for our purposes. We have the following lemma.

\begin{lemma}
\label{almostpaulsenofm2}
  Let $S \subset \bigoplus_{k=1}^3 M_2(\calB)$ be the operator system defined by
  \begin{align*}
  \calS:=\left\{ \bigoplus_{k=1}^3 \left[\begin{array}{cc|cc}
  a_k & & b_k & c_k \\
  & a_k & c_k & b_k \\ \hline
  d_k & f_k & a_k & \\
  f_k & d_k & & a_k \end{array}\right]: a_k=a_{\ell}, \, b_k=b_{\ell}, \, d_k=d_{\ell}, \, \text{for all } 1 \leq k,\ell \leq 3 \right\}
  \end{align*}
  Then $\calS$ is hyperrigid in its $C^{\ast}$-cover $M_2(\calB) \oplus M_2(\calB) \oplus M_2(\calB)$.
\end{lemma}

\begin{proof}
Let $\pi:\bigoplus_{k=1}^3 M_2(\calB)=(\bigoplus_{k=1}^3 M_2) \otimes \calB \to B(H)$ be an irreducible representation.  Then up to unitary equivalence, we may assume that $H=K \otimes L$ and $\pi=\rho \otimes \sigma$, where $\rho:M_2 \oplus M_2 \oplus M_2 \to B(K)$ and $\sigma:\calB \to B(L)$ are $*$-homomorphisms. Since $\rho \otimes \sigma$ is irreducible and $\sigma(\calB)$ is abelian, we must have $\dim(\sigma(\calB))=1$ and $L=\C$. Therefore, we may identify $H=K$ and $\sigma \left(\begin{bmatrix} 0 & 1 \\ 1 & 0 \end{bmatrix} \right)=\omega I_H$, where $\omega \in \{1,-1\}$.  Since $\pi$ is irreducible, it must be surjective.  Hence, $\rho$ is surjective.  Thus, $H$ is finite-dimensional and we may write $B(H)=M_D$ for some $D \in \mathbb{N}$. Let $\rho_k:M_2 \to M_D$ be the restriction of $\rho$ to the $k$-th summand of $M_2$ in $M_2 \oplus M_2 \oplus M_2$. Each $\rho_k$ is a $*$-homomorphism, so since $M_2$ is simple, $\rho_k$ is either injective or the zero map.  If at two of the $\rho_k$'s were injective (say, $\rho_1$ and $\rho_2$) and $T_1 \oplus T_2 \in M_2 \oplus M_2$ were such that $\rho(T_1 \oplus T_2 \oplus 0)=0$, then we would have
$$0=\rho(T_1 \oplus T_2 \oplus 0)^*\rho(T_1 \oplus T_2 \oplus 0)=\rho_1(T_1^*T_1)+\rho_2(T_2^*T_2).$$
By injectivity of each $\rho_k$, we must have $T_1=T_2=0$, so that $\rho_{|M_2 \oplus M_2 \oplus 0}$ is injective.  But then this restriction would be an isomorphism of $M_2 \oplus M_2$ onto $M_D$, with the latter being simple, which is a contradiction.  Similarly, it is impossible to have all three $\rho_k$'s injective. Moreover, since $I_D=\rho(I_2 \oplus I_2 \oplus I_2)=\sum_{k=1}^3 \rho_k(I_2)$, we must have that exactly one $\rho_k$ is non-zero. Therefore, we may assume without loss of generality that $\rho_1$ is non-zero (and hence an isomorphism), while $\rho_2=\rho_3=0$. In particular, we may assume that $D=2$.

Up to unitary conjugation, we may assume that $\rho_1=\id_{M_2}$. Let $\varphi$ be the restriction of $\pi$ to $S$, and let $\psi:\bigoplus_{k=1}^3 M_2(\calB) \to M_2$ be any ucp extension of $\varphi$.  Let $\psi=V^* \Pi(\cdot)V$ be a minimal Stinespring representation of $\psi$ on some Hilbert space $H_{\Pi}$.  We consider $\bigoplus_{k=1}^3 M_2(\calB)$ as a subalgebra of $\bigoplus_{k=1}^3 M_4$.

Let $X=E_{14}+E_{23} \in S$. We note that $\psi(X)=\pi(X)=\omega E_{12} \in M_2$. With respect to the decomposition $H_{\Pi}=\ran(V) \oplus \ran(V)^{\perp}$, for $i,j,k,\ell \in \Lambda$, we have
$$\Pi(X)=\begin{bmatrix} \varphi(X) & A \\ B & C \end{bmatrix}=\begin{bmatrix} \omega E_{12} & A \\ B & C \end{bmatrix},$$
for some operators $A,B,C$. Noting that $X^*X+XX^*=\sum_{i=1}^4 E_{ii}$, we have
$$\Pi \left(\sum_{i=1}^4 E_{ii} \right)=\Pi(X^*X+XX^*)=\Pi(X)^*\Pi(X)+\Pi(X)\Pi(X)^*=\begin{bmatrix} I+B^*B+AA^* & * \\ * & * \end{bmatrix}.$$
Thus, $\varphi\left(\sum_{i=1}^4 E_{ii}\right)=I+B^*B+AA^* \leq \varphi(I_{12})=I_2$. This forces $B^*B+AA^*=0$, so that $A=0$ and $B=0$. Considering the $(1,1)$-block, it follows that $\psi(X^*X)=E_{22}=\psi(X)^*\psi(X)$ and $\psi(XX^*)=E_{11}=\psi(X)\psi(X)^*$. Thus, $X$ lies in the multiplicative domain $\calM_{\psi}$ of $\psi$. Moreover, $\psi(E_{ii})=0$ for all $5 \leq i \leq 12$. It readily follows that $\psi_{|0 \oplus M_2(\calB) \oplus M_2(\calB)}=0$. Replacing $X$ with $Y=E_{13}+E_{24}$, it is easy to see that $Y \in \mathcal{M}_{\psi}$ as well, while $\psi(Y)=E_{12}$. Let $W=E_{12}+E_{21}$ and $Z=E_{34}+E_{43}$. Then the first copy of $M_2(\calB)$ is generated as a $C^*$-algebra by the four elements $X,Y,Z,W$. Since we have $W^*W=E_{11}+E_{22}$ and $Z^*Z=E_{33}+E_{44}$, we need only show that $\psi(W)=\pi(Z)$ and $\psi(Z)=\pi(Z)$.  If this assertion holds, then $W$ and $Z$ would lie in $\mathcal{M}_{\psi}$, from which it would follow that $M_2(\calB) \oplus 0 \oplus 0 \subseteq \mathcal{M}_{\psi}$ and $\pi=\psi$.
Since $W=XY^*$, we may write
$$\psi(W)=\psi(XY^*)=\psi(X)\psi(Y)^*=\omega E_{12}E_{21}=E_{11}.$$
Similarly, since $Z=X^*Y$, we have
$$\psi(Z)=\psi(X^*Y)=\psi(X)^*\psi(Y)=\omega E_{21}E_{12}=\omega E_{22}.$$
It readily follows that $\psi(Z)=\pi(Z)$ and $\pi(W)=\psi(W)$, while $Z,W \in \mathcal{M}_{\psi}$.  Therefore, $\psi=\pi$. Applying Proposition \ref{proposition: hyperrigid passes from Paulsen to system} completes the proof.
\end{proof}

\begin{lemma}\label{lemma: Paulsen of Paulsen}
  The operator sub-algebra of $\bigoplus_{k=1}^3 M_2(\calB)$ given by
\begin{align*}
\calA:= \left\{ \bigoplus_{k=1}^3 \left[\begin{array}{cc|cc}
a_k & & b_k & c_k \\
& a_k & c_k & b_k \\ \hline
& & a_k & \\
& & & a_k \end{array}\right]: a_k=a_{\ell}, \, b_k=b_{\ell}, \, \forall 1 \leq k,\ell \leq 3\right\}
\end{align*}
  is hyperrigid in $M_2(\calB) \oplus M_2(\calB) \oplus M_2(\calB)$.
\end{lemma}

\begin{proof}
The operator system $\calS$ in Lemma \ref{almostpaulsenofm2} is precisely $\calA+\calA^*$ in the $C^*$-cover $\bigoplus_{k=1}^3 M_2(\calB)$.  The desired result follows by Lemma \ref{hyperrigidopalgiffaplusastar}.
\end{proof}

We can now obtain hyperrigidity of the nuclearity detectors $\calW_{3,2}$ and $\calU(\calW_{3,2})$.
\begin{theorem}
\label{Theorem: hyperrigidity of counterexample}
  The operator algebra $\calU(\calW_{3,2})$ is hyperrigid in $C_{\env}^*(\calU(\calW_{3,2}))=M_2\left( \bigoplus_{k=1}^3\calB \right)$, and $\calW_{3,2}$ is hyperrigid in $C_{\env}^*(\calW_{3,2})=\bigoplus_{k=1}^3 \calB$.
\end{theorem}

\begin{proof}
Since $\calW_{3,2}$ is the set of all matrices in $\calB \oplus \calB \oplus \calB$ of the form $$\bigoplus_{k=1}^3 \begin{bmatrix} a_k & b_k \\ b_k & a_k \end{bmatrix},$$
where $a_k=a_{\ell}$ for all $1 \leq k,\ell \leq 3$, it follows that $C^*(\calW_{3,2}) \subseteq \oplus_{k=1}^3 \calB$.  On the other hand, for each $k$, the element $V_k$ defined by $\begin{bmatrix} 0 & 1 \\ 1 & 0 \end{bmatrix}$ in the $k$-th summand and $0$ in the other summands is an element of $\calW_{3,2}$, and $V_k^2=V_k^*V_k$ is $I_2$ in the $k$-th summand and $0$ otherwise.  It follows that $C^*(\calW_{3,2})=\oplus_{k=1}^3 \calB$. Using Proposition \ref{proposition: cover for S gives cover for U(S)}, $M_2(\oplus_{k=1}^3 \calB)$ is a $C^*$-cover of $\calU(\calW_{3,2})$.

To establish hyperrigidity of $\calU(\calW_{3,2})$ in $M_2 \left( \bigoplus_{k=1}^3 \calB \right)$, let $\pi$ be the restriction of the canonical shuffle $M_2(M_6) \simeq M_6(M_2)$ to $M_2 \left( \bigoplus_{k=1}^3 \calB \right)$. Then $\pi$ is a $*$-isomorphism from $M_2 \left(\bigoplus_{k=1}^3 \calB \right)$ onto $\oplus_{k=1}^3 M_2(\calB)$ that sends $\calU(\calW_{3,2})$ onto the operator algebra $\calA$ in Lemma \ref{lemma: Paulsen of Paulsen}. Applying Lemma \ref{lemma: uep properties}, since $\calA$ is hyperrigid in $\oplus_{k=1}^3 M_2(\calB)$, $\calU(\calW_{3,2})$ is hyperrigid in $M_2 \left(\bigoplus_{k=1}^3 \calB \right)$.  The analogous claim for $\calW_{3,2}$ follows by Proposition \ref{proposition: hyperrigid passes from Paulsen to system}. The claim about $C^*$-envelopes immediately follows by Corollary \ref{corollary: hyperrigidity of U(S) in M2(D) tells cstar envelope of S and U(S)}.
\end{proof}

\begin{remark}
  If an operator algebra $\mathcal{A}$ is hyperrigid in $C_{env}^*(\mathcal{A})$, then a theorem of Katsoulis \cite[Theorem 2.7]{katsoulis17} shows that $C^*_{env}(\calA \rtimes_{C^*_{env}(\calA),\alpha} G) = C^*_{env}(\calA) \rtimes_{\alpha} G$. Corollary \ref{KRproblem1} shows that the analogous question for the full crossed product can fail even if the operator algebra $\calA$ is hyperrigid in its $C^*$-envelope. What about the case when $\calA$ is not hyperrigid in $C^*_{env}(\calA)$? What about the operator system case? Is the hyperrigidity of an operator system $\calS$ equivalent to the hyperrigidity of $\calU(\calS)$?
\end{remark}

\section*{Acknowledgements}

We would like to thank Ken Davidson, Matt Kennedy, and Vern Paulsen for their careful reading of this paper and their comments. We thank the referee for their helpful suggestions, which have improved the readability of the paper. We would also like to thank Chris Schafhauser for valuable conversations. The second author was supported by an NSERC Postgraduate Scholarship.


\begin{thebibliography}{99}
\bib{ACKKLS}{article}{
author={M. Argerami},
author={S. Coskey},
author={M. Kalantar},
author={M. Kennedy},
author={M. Lupini},
author={M. Sabok},
title={The classification problem for finitely generated operator systems and spaces},
journal={preprint (arXiv:1411.0512 [math.OA])},
year={2014},
}
\bib{arveson69}{article}{
author={W. Arveson},
title={Subalgebras of $C^{\ast}$-algebras},
journal={Acta Mathematica},
volume={123},
year={1969},
pages={141--224},
}
\bib{arveson03}{unpublished}{
author={W. Arveson},
title={Notes on the unique extension property},
year={preprint (2006)},
note={Available at https://math.berkeley.edu/~arveson/Dvi/unExt.pdf},
}
\bib{arveson08}{article}{
author={W. Arveson},
title={The non-commutative Choquet boundary},
journal={Journal of the American Mathematical Society},
volume={21},
number={4},
year={2008},
pages={1065--1084},
}
\bib{arveson11}{article}{
author={W. Arveson},
title={The non-commutative Choquet boundary II: hyperrigidity},
journal={Israel Journal of Mathematics},
volume={184},
number={1},
year={2011},
pages={349--385},
}
\bib{BKQR}{article}{
author={E. Bedos},
author={S. Kaliszewski},
author={J. Quigg},
author={D. Robertson},
title={A new look at crossed product correspondences
and associated $C^{\ast}$–algebras}, 
journal={Journal of Mathematical Analysis and Applications},
volume={426},
number={2},
year={2015},
pages={1080–-1098},
}
\bib{blecherduncan}{article}{
author={D.P. Blecher},
author={B.L. Duncan},
title={Nuclearity-related properties for nonselfadjoint algebras},
journal={Journal of Operator Theory},
volume={65},
number={1},
year={2011},
pages={47--70},
}
\bib{blecherlemerdy}{book}{
author={D.P. Blecher},
author={C. Le Merdy},
title={Operator algebras and their modules--an operator space approach},
series={London Mathematical Society Monographs. New Series},
volume={30},
publisher={The Clarendon Press, Oxford University Press, Oxford},
date={2004},
pages={x+387},
}
\bib{blecherroyce}{article}{
author={D.P. Blecher},
author={M. Royce},
title={Extensions of operator algebras I},
journal={Journal of Mathematical Analysis and Applications},
volume={339},
number={2},
year={2008},
pages={1451--1467},
}
\bib{BrownAndOzawa}{book}{
  author={N.P. Brown},
  author={N. Ozawa},
  title={$C^{\ast}$-algebras and finite-dimensional approximations},
  publisher={American Mathematical Society, Providence, RI},
  series={Graduate Studies in Mathematics},
  volume={88},
  year={2008},
  pages={xv+509},
}
\bib{choieffrosopsys}{article}{
author={M.-D. Choi},
author={E.G. Effros},
title={Injectivity and operator spaces},
journal={Journal of Functional Analysis},
volume={24},
number={2},
year={1977},
pages={156--209},
}
\bib{choieffros}{article}{
author={M.-D. Choi},
author={E.G. Effros},
title={Nuclear $C^*$-algebras and injectivity: the general case},
journal={Indiana University Mathematics Journal},
volume={26},
number={3},
year={1977},
pages={443--446},
}
\bib{dritschelmccullough}{article}{
author={M.A. Dritschel},
author={S.A. Mccullough},
title={Boundary representations for families of representations of operator algebras and spaces},
journal={Journal of Operator Theory},
volume={53},
number={1},
year={2005},
pages={159--167},
}
\bib{DiscGroup}{article}{
author={D. Farenick},
author={A.S. Kavruk},
author={V.I. Paulsen},
author={I.G. Todorov},
title={Operator systems from discrete groups},
journal= {Communications in Mathematical Physics},
volume= {329},
number = {1},
year = {2014},
pages = {207--238},
}
\bib{FP}{article}{
author={D. Farenick},
author={V.I. Paulsen},
title={Operator system quotients of matrix algebras and their tensor products},
journal={Mathematica Scandinavica},
volume={111},
number={2},
pages={210--243},
date={2012},
}
\bib{FSW}{article}{
author={B.E. Forrest},
author={N. Spronk},
author={M. Wiersma},
title={Existence of tracial states on reduced group $C^{\ast}$-algebras},
journal={preprint (arXiv:1706.05354 [math.OA])},
year={2017},
}
\bib{guptaluthra}{article}{
author={V.P. Gupta},
author={P. Luthra},
title={Operator system nuclearity via $C^{\ast}$-envelopes},
journal={Journal of the Australian Mathematical Society},
volume={101},
number={3},
year={2016},
pages={356--375},
}
\bib{hamanaopsys}{article}{
author={M. Hamana},
title={Injective envelopes of operator systems},
journal={Publications of the Research Institute for Mathematical Sciences, Kyoto University},
volume={15},
number={3},
year={1979},
pages={773--785},
}
\bib{hamanadynsys}{article}{
author={M. Hamana},
title={Injective envelopes of $C^{\ast}$-dynamical systems},
journal={Tohoku Mathematical Journal (2)},
volume={37},
number={4},
year={1985},
pages={463--487},
}
\bib{hao-ng}{article}{
author={G. Hao},
author={C.-K. Ng},
title={Crossed products of $C^{\ast}$-correspondences by amenable group actions},
journal={Journal of Mathematical Analysis and Applications},
volume={345},
number={2},
year={2008},
pages={702--707},
}
\bib{kalantarkennedy}{article}{
author={M. Kalantar},
author={M. Kennedy},
title={Boundaries of reduced $C^{\ast}$-algebras of discrete groups},
journal={Journal f{\"u}r die reine und angewandte Mathematik},
volume={2017},
number={727},
year={2017},
pages={247--267},
}
\bib{KQR}{article}{
author={S. Kaliszewski},
author={J. Quigg},
author={D. Robertson},
title={Functoriality of Cuntz-Pimsner correspondence
maps},
journal={Journal of Mathematical Analysis and Applications},
volume={405},
number={1},
year={2013},
pages={1-–11},
}
\bib{katsoulisramsey}{article}{
author={E.G. Katsoulis},
author={C. Ramsey},
title={Crossed products of operator algebras},
journal={Memoirs of the American Mathematical Society},
year={in press},
}
\bib{katsoulisramsey18}{article}{
author={E.G. Katsoulis},
author={C. Ramsey},
title={A non-selfadjoint approach to the Hao-Ng isomorphism problem},
journal={preprint (arXiv:1807.11425 [math.OA])},
year={2018},
}
\bib{katsoulis17}{article}{
author={E.G. Katsoulis},
title={$C^{\ast}$-envelopes and the Hao-Ng isomorphism for discrete groups},
journal={International Mathematics Research Notices},
volume={2017},
number={18},
year={2017},
pages={5751--5768},
}
\bib{kavruk14}{article}{
author={A.S. Kavruk},
title={Nuclearity related properties in operator systems},
journal={Journal of Operator Theory},
volume={71},
number={1},
year={2014},
pages={95--156},
}
\bib{kavruk15}{article}{
author={A.S. Kavruk},
title={On a non-commutative analogue of a classical result of Namioka and Phelps},
journal={Journal of Functional Analysis},
volume={269},
number={10},
year={2015},
pages={3282--3303},
}
\bib{KPTT}{article}{
author={A.S. Kavruk},
author={V.I. Paulsen},
author={I.G. Todorov},
author={M. Tomforde},
title={Tensor products of operator systems},
journal={Journal of Functional Analysis},
volume={261},
number={2},
year={2011},
pages={267--299},
}
\bib{KPTTquotients}{article}{
author={A.S. Kavruk},
author={V.I. Paulsen},
author={I.G. Todorov},
author={M. Tomforde},
title={Quotients, exactness, and nuclearity in the operator system category},
journal={Advances in Mathematics},
volume={235},
year={2013},
pages={321--360},
}
\bib{kawabe}{article}{
author={T. Kawabe},
title={Uniformly recurrent subgroups and the ideal structure of reduced crossed products},
journal={preprint (arXiv:1701.03413 [math.OA])},
year={2017},
}
\bib{quasifree}{incollection}{
author={A. Kishimoto},
author={A. Kumjian},
title={Crossed products of Cuntz algebras by quasi-free automorphisms},
booktitle={Operator algebras and their applications (Waterloo, ON, 1994/1995)},
series={Fields Institute Communications},
volume={13},
publisher={American Mathematical Society, Providence, RI},
year={1997},
pages={173--192},
}
\bib{paulsenbook}{book}{
author={V.I. Paulsen},
title={Completely bounded maps and operator algebras},
series={Cambridge Studies in Advanced Mathematics},
volume={78},
publisher={Cambridge University Press, Cambridge},
year={2002},
pages={xii+300},
}
\bib{danawilliams}{book}{
author={D.P. Williams},
title={Crossed products of $C^{\ast}$-algebras},
series={Mathematical Surveys and Monographs},
volume={134},
publisher={American Mathematical Society, Providence, RI},
year={2007},
pages={xvi+528},
}
\end{thebibliography}
\end{document}